\documentclass[12pt,a4paper]{amsart}
\usepackage{amsmath, amssymb, amsfonts,enumerate}

\usepackage{latexsym,amsbsy,amsmath,amscd,amssymb,enumerate}  

\usepackage{amsthm}       
\usepackage{amssymb,amsmath,latexsym}
 \newtheorem{definition}{Definition} [section]       
 \newtheorem{remark}[definition]{Remark}

 \newtheorem{problem}[definition]{Problem}  
 \newtheorem{comment}[definition]{Comment}

 \newtheorem{proposition}[definition]{Proposition}       
 \newtheorem{theorem}[definition]{Theorem}       
 \newtheorem{corollary}[definition]{Corollary}       
  \newtheorem{lemma}[definition]{Lemma}

\numberwithin{equation}{section}

\newcommand{\Ind}{\mbox{\rm Ind}}

\newcommand{\Res}{\mbox{\rm Res}}

\newcommand{\Stab}{\mbox{\rm Stab}}
\newcommand{\Sym}{\mbox{\rm Sym}}

\newcommand{\Hom}{\mbox{\rm Hom}}

\def\NN{{\mathbb{N}}}
\def\RR{{\mathbb{R}}}

\def\ZZ{{\mathbb{Z}}}

\def\FF{{\mathbb{F}}}

\def\CC{{\mathbb{C}}}

\def\C{{\mathbb{C}}}

\def\HomS{\text{\rm Hom}^\text{\rm Sym}}
\def\HomA{\text{\rm Hom}^\text{\rm Skew}}

\renewcommand{\dim}{\mbox{\rm dim}}
\newcommand{\inv}{\text{\rm inv}}
\newcommand{\Id}{\mbox{\rm Id}} 
\newcommand{\Aut}{\mbox{\rm Aut}} 
\newcommand{\Sp}{\mbox{\rm Sp}} 
\newcommand{\GL}{\mbox{\rm GL}} 
\newcommand{\PSL}{\mbox{\rm PSL}} 
\newcommand{\SL}{\mbox{\rm SL}} 
\newcommand{\SU}{\mbox{\rm SU}} 
\newcommand{\SO}{\mbox{\rm SO}} 
\newcommand{\OO}{\mbox{\rm O}} 
\newcommand{\U}{\mbox{\rm GU}} 
\newcommand{\GU}{\mbox{\rm GU}}

\newcommand{\CL}{{\mathbb{CL}}}

\begin{document}

\title{Mackey's theory of $\tau$-conjugate representations for finite groups} 
\author{Tullio Ceccherini-Silberstein}
\address{Dipartimento di Ingegneria, Universit\`a del Sannio, C.so
Garibaldi 107, 82100 Benevento, Italy}
\email{tceccher@mat.uniroma3.it}
\author{Fabio Scarabotti}
\address{Dipartimento SBAI, Sapienza Universit\`a  di Roma, via A. Scarpa 8, 00161 Roma, Italy}
\email{fabio.scarabotti@sbai.uniroma1.it}
\author{Filippo Tolli}
\address{Dipartimento di Matematica e Fisica, Universit\`a Roma TRE, L. San Leonardo Murialdo 1, 00146  Roma, Italy}
\email{tolli@mat.uniroma3.it}

\begin{abstract}
The aim of the present paper is to expose two contributions of Mackey, together with a more recent result of Kawanaka and Matsuyama, generalized by Bump and Ginzburg, on the representation theory of a finite group
equipped with an involutory anti-automorphism (e.g. the anti-automorphism $g\mapsto g^{-1}$).
Mackey's first contribution is a detailed version of the so-called Gelfand criterion for weakly symmetric Gelfand pairs. Mackey's second contribution is a characterization of simply reducible groups
(a notion introduced by Wigner).
The other result is a twisted version of the Frobenius-Schur theorem, where ``twisted" refers to the above-mentioned involutory anti-automorphism.
\end{abstract}
\date{\today}
\maketitle 

\tableofcontents
\section{Introduction}
Finite Gelfand pairs not only constitute a useful tool for analyzing a wide range of problems ranging from combinatorics, to orthogonal
polynomials and to stochastic processes, but may also be used to shed light into theoretical problems of representation theory. The simplest example is provided by the possibility to recast the decomposition of the group algebra of a given finite group $G$, together with the associated harmonic analysis, by using the action on $G$ of the direct product $G\times G$. Another example comes from the application of Gelfand pairs in the theory of multiplicity free groups, a key tool in the recent approach of Okounkov and Vershik to the representation theory of the symmetric groups \cite{OV1,OV2} (see also \cite{book2}).

Let $G$ be  a finite group.
Recall that the \emph{conjugate} of a (unitary) representation $(\rho,V)$ of  $G$,
is the $G$-representation $(\rho',V')$ where  $V'$ is the dual of $V$ and $[\rho'(g)v'](v) = v'[\rho(g^{-1})v]$
for all $g \in G$, $v \in V$, and $v'\in V'$. One then says that $\rho$ is \emph{self-conjugate} provided
$\rho \sim \rho'$; this is in turn equivalent to the associated character $\chi_\rho$ being real-valued. When $\rho$ is not self-conjugate, one says that it is {\it complex}.
The class of self-conjugate $G$-representations splits into two subclasses according to the associated matrix
coefficients of the representation $\rho$ being real-valued or not: in the first case, one says that $\rho$ is
\emph{real}, in the second case $\rho$ is termed \emph{quaternionic}.

Now let $K\leq G$ be a subgroup and denote by $X = G/K$ the corresponding homogeneous space of left cosets of $K$ in $G$. Setting $L(X) = \{f:X \to \CC\}$, denote by $(\lambda,L(X))$ the corresponding \emph{permutation representation} 
defined by $[\lambda(g)f](x) = f(g^{-1}x)$ for all $g \in G$ and $f \in L(X)$. Recall that $(G,K)$ is a \emph{Gelfand
pair} provided the permutation representation $\lambda$ decomposes \emph{multiplicity-free}, that is,
\begin{equation}
\label{eq:L-X-intro}
\lambda = \bigoplus_{i \in I} \rho_i
\end{equation} 
with $\rho_i \not\sim \rho_j$ for $i \neq j$. It is well known that 
if $g^{-1} \in KgK$ for all $g \in G$, then $(G,K)$ is a Gelfand pair; in this case all representations $\rho_i$ in \eqref{eq:L-X-intro} are real, and one then says that $(G,K)$ is \emph{symmetric}. This last terminology
is due to the fact that the $G$-orbits on $X \times X$ under the diagonal action are symmetric (with respect
to the flip $(x_1,x_2) \mapsto (x_2,x_1)$, $x_1,x_2 \in X$).

A remarkable classical problem in representation theory is to determine the decomposition of the tensor
product of two (irreducible) representations. In particular, one says that $G$ is \emph{simply reducible} if
(i) $\rho_1 \otimes \rho_2$ decomposes multiplicity free for all irreducible $G$-representations $\rho_1$ and $\rho_2$
and (ii) every irreducible $G$-representation  is self-conjugate.

The class of simply reducible groups was introduced by E. Wigner \cite{wigner} in his research on group representations and quantum mechanics. This notion is quite useful since many of the symmetry groups one encounters in atomic and molecular systems are simple reducible, and algebraic manipulations of tensor operators become much easier for such groups. Wigner wrote:``The groups of most eigenvalue problems occurring in quantum theory are S.R'' (where ``S.R.'' stands for ``simply reducible'') having in mind the study of ``small perturbation'' of the 
``united system'' of two eigenvalue problems invariant under some group $G$ of symmetries. Then simple reducibility guarantees that the characteristic functions of the eigenvalues into which the united system splits can be determined in ``first approximation'' by the invariance of the eigenvalue problem under $G$. This is the case, for instance, for the \emph{angular momentum} in quantum mechanics.
We mention that the multiplicity-freeness of the representations in the definition of simply reducible groups is the condition for the validity of the well known \emph{Eckart-Wigner theorem} in quantum mechanics.
Also, an important task in spectroscopy is to calculate matrix elements in order to determine energy spectra and transition intensities. One way to incorporate symmetry considerations connected to a group $G$ or rather a pair 
$(G,H)$ of groups, where $H \leq G$, is to use the \emph{Wigner-Racah} calculus associated with the inclusion under consideration: this is generally understood as the set of algebraic manipulations  concerning the coupling and the coupling coefficients for the group $G$. The Wigner-Racah calculus was originally developed for simply reducible groups 
\cite{racah1, racah2, wignerbook, wigner-matrices} and, later, for some other groups of interest in nuclear, atomic, molecular, condensed matter physics \cite[Chapter 5]{hamermesh} as well as in quantum chemistry \cite{griffiths}. 

Returning back to purely representation theory, Wigner \cite{wigner} listed the following examples of simply reducible groups: the symmetric groups $S_3$ ($\cong D_3$) and $S_4$ ($\cong T_h$),  the quaternion
group $Q_8$ and the \emph{rotational groups} $\OO(3)$,  $\SO(3)$ or $\SU(2)$. More generally, it is nowdays known (cf. \cite[Appendix 3.A]{clementi}) that most of the \emph{molecular symmetry groups} such as (using Schoenflies notation)
$D_{\infty h}, C_{\infty v}, C_{2 v}, C_{3, v}, C_{2 h}, D_{3 h}, D_{3 d}, D_{6 h}, T_d$ and $O_h$ are simply reducible.
On the other hand, the \emph{icosahedral group} $I_h$ is not simply reducible, although it possesses only real characters.

In the automorphic setting, Prasad \cite{prasad} implicitly showed that if $k$ is a local field then the (infinite)
group $G = \GL(2,k)$ is simply reducible. Indeed, he proved that the number of $G$-invariant linear forms
on the tensor product of three admissible representations of $G$ is at most one (up to scalars).
This is also discussed in Section 10 of the survey article by Gross (Prasad's advisor) \cite{gross}, 
on Gelfand pairs and their applications to number theory.

We also mention that simply reducible groups are of some interest also in the theory of \emph{association schemes}
(see \cite[Chapter 2]{BI}).

As pointed out by A.I. Kostrikin in \cite{kostrikin}, there is no complete description of all simply reducible groups. Strunkov investigated simple reducibility in \cite{strunkov} and  suggested (cf.  \cite[Problem 11.94]{KO} in the Kourovka notebook) that the simply reducible groups must be solvable.
After some partial results by Kazarin and Yanishevski$\breve{{\rm \i}}$ \cite{kazarin-Y}, this conjecture was settled by Kazarin and Chankov \cite{KC}.

Wigner \cite{wigner} gave a curious criterion for simply reducibility. He showed that, denoting by $v(g) = |\{h \in G: hg=gh\}|$ (resp. $\zeta(g) = |\{h \in G: h^2 = g\}|$) the cardinality of the centralizer (resp. the number of square roots) of an element $g \in G$, then the equality
\begin{equation}
\label{eq:wigner-intro}
\sum_{g \in G} \zeta(g)^3 = \sum_{g \in G} v(g)^2
\end{equation} 
holds if and only if $G$ is simply reducible. 

A fundamental theorem of Frobenius and Schur \cite{FS} provides a criterion for determining the type of a given 
irreducible representation $\rho$, namely
\begin{equation}
\label{eq:Frobenius-introduction}
\frac{1}{|G|}\sum_{g \in G} \chi_\rho(g^2) = 
\begin{cases} 
1 & \mbox{ if $\rho$ is real}\\
-1 & \mbox{ if $\rho$ is quaternionic}\\
0 & \mbox{ if $\rho$ is complex}
\end{cases}
\end{equation}
see, for instance, \cite[Theorem 9.7.7]{book}.
Moreover, the number $h$ of pairwise inequivalent irreducible self-conjugate $G$-representations is given by
\begin{equation}
\label{eq:number-sc-introduction}
h = \frac{1}{|G|}\sum_{g \in G} \zeta(g)^2
\end{equation}
(cf.
\cite[Theorem 9.7.10]{book}).

In this research-expository paper, following Mackey \cite{mackey-mf}, Kawanaka and Matsuyama \cite{kawanaka}, and
Bump and Ginzburg \cite{bump-ginzburg}, we endow $G$ with an involutory anti-automorphism $\tau \colon G \to G$.
Mackey in \cite{mackeysym} originally analyzed only the case when $\tau$ is the anti-automorphism $g\mapsto g^{-1}$ and then, in \cite{mackey-mf}, generalized his results by considering any involutory anti-automorphism. The proofs are even simpler but heavily rely on \cite{mackeysym} (the reader cannot read the second paper without having at hand the first one). Here we give a complete and self-contained treatment of all principal results in \cite{mackey-mf}, providing more details and using modern notation. 

We then present in Theorem \ref{t1TFS} (Twisted Frobenius-Schur theorem) the main result of
Kawanaka and Mastuyama in \cite{kawanaka}. Our proof follows the lines indicated in Bump's monograph  \cite[Exercise 4.5.1]{Bump} but also heavily uses the powerful machinery of A.H. Clifford theory specialized for subgroups of index two (see, for instance, \cite[Section 3]{Clifford}). Note that Bump and Ginzburg \cite{bump-ginzburg} consider further generalizations involving anti-automorphisms of finite order (i.e. not necessarily involutive). 

Let $\tau: G \to G$ be an involutory anti-automorphism.

Given a $G$-representation $(\rho,V)$, we then define its $\tau$-\emph{conjugate} as the $G$-representation $(\rho^\tau, V')$ defined by setting $[\rho^\tau(g)v'](v) = v'[\rho(\tau(g))v]$ for all $g \in G$, $v \in V$, and $v'\in V'$. Then we introduce (cf. \cite{kawanaka}) the associated $\tau$-\emph{Frobenius-Schur number}
(or $\tau$-\emph{Frobenius-Schur indicator}) $C_\tau(\rho)$ defined by
\[
C_\tau(\rho)=\dim\HomS_G(\rho^\tau,\rho)-\dim\HomA_G(\rho^\tau,\rho)
\]
where $\HomS_G$ (resp. $\HomA_G$) denotes the space of symmetric (resp. antisymmetric) intertwining operators,
and show that, if $\rho$ is irreducible, it may take only the three values $1$, $-1$, and $0$.

Given a subgroup $K$, we consider the $\tau$-\emph{conjugate} $\lambda^\tau$ of the associated permutation representation. Suppose that $\lambda^\tau \sim \lambda$ (note that this is always the case if $K$ is $\tau$-invariant,
i.e., $\tau(K) = K$), then we present a characterization (the Mackey-Gelfand criterion, see Theorem \ref{t:3MF}) of the corresponding analogue of ``symmetric Gelfand pair'' that we recover as a particular case.

We say that $G$ is \emph{$\tau$-simply reducible} provided (i) $\rho_1 \otimes \rho_2$ is multiplicity-free 
and (ii) $\rho^\tau \sim \rho$, for all irreducible $G$-representations $\rho_1, \rho_2$ and $\rho$.
We then present the Mackey criterion (Theorem \ref{t:4SRG1}) and the Mackey-Wigner criterion (Corollary \ref{c:10SR2})
for $\tau$-simple reducibility, a generalization of Wigner's original criterion we alluded to above (cf. \eqref{eq:wigner-intro}); the latter is expressed in terms of the equality 
\[
\sum_{g \in G} \zeta_\tau(g)^3 = \sum_{g \in G} v(g)^2.
\]

As an application of both the Mackey criterion and the Mackey-Wigner criterion, we present new examples of
$\tau$-simply reducible groups: in Section \ref{s:Clifford} we show that the (W.K.) Clifford groups $\CL(n)$ 
are $\tau$-simply reducible (where the involutive anti-automorphism $\tau$ of $\CL(n)$ is suitably defined according to the congruence class of $n$ modulo $4$).

Generalizing the characterization \eqref{eq:Frobenius-introduction}, we show (the Kawanaka and Matsuyama theorem 
(Theorem \ref{t1TFS}))
that 
\[
C_\tau(\rho) = \frac{1}{|G|}\sum_{g\in G}\chi_\rho(\tau(g)^{-1}g).
\]

Finally, in the last section, we present a twisted Frobenius-Schur type theorem in the context of Gelfand pairs (Theorem \ref{t:TFSGP}). This result, together with the ones on $\tau$-simple reducibility of the Clifford groups we alluded to above, constitutes our original contribution to the theory.\\

\section{Preliminaries and notation}
\label{s:preliminaries}
\subsection{Linear algebra}
\label{ss:linear-algebra}
In order to fix notation, we begin by recalling some elementary notions of linear algebra.
Let $V,W$ be finite dimensional complex vector spaces and denote by $V',W'$ their duals. 
We denote by $\Hom(V,W)$ the space of all linear operators $A \colon V\to W$. 

If $A\in\Hom(V,W)$ its {\em transpose} is the linear operator $A^T:W'\to V'$ defined by setting 
\[
(A^Tw')(v)=w'(Av)
\]
for all $w'\in W'$ and $v\in V$. Let $Z$ be another finite dimensional complex vector space and suppose that
$B\in\Hom(V,W)$ and $A\in\Hom(W,Z)$. Then, it is immediate to check that $(AB)^T=B^TA^T$. Moreover, modulo the
canonical identification of $V$ and its bidual $V'' = (V')'$ (this is given by $v \leftrightarrow v''$ where
$v'' \in V''$ is defined by $v''(v') = v'(v)$ for all $v' \in V'$), we have $(A^T)^T = A$.
Given a basis $\{v_1,v_2,\dotsc, v_n\}$ in $V$, we denote by $\{v_1',v_2',\dotsc, v_n'\}$ the corresponding dual basis
of $V'$ which is defined by the conditions $v_j'(v_i)=\delta_{i,j}$ for $i,j=1,2,\dotsc,n$. 
Let now $\{w_1, w_2, \ldots, w_m\}$ be a basis for $W$. Let $M_A = (a_{ki})_{\substack{k=1,2,\ldots,m\\i=1,2,\ldots,n}}$ the matrix associated with the linear operator $A$, that is,  $Av_i=\sum_{k=1}^ma_{ki}w_k$, for all $i=1,2,\ldots,n$. Then $a_{ki}= w_k'(Av_i)$ and $A^Tw_k'=\sum_{i=1}^na_{ki}v_i'$;
in other words, the matrix $M_{A^T}$ associated with the transpose operator $A^T$ equals the transpose $(M_A)^t = (a_{ik})_{\substack{i=1,2,\ldots,n\\k=1,2,\ldots,m}}$ of the matrix $M_A$ associated with $A$.

Suppose now that $V$ is endowed with a hermitian scalar product denoted $\langle\cdot,\cdot\rangle_V$. 
The associated {\em Riesz map} is the antilinear bijective map $\theta_V \colon V \to V'$ defined by setting
\[
(\theta_V v)(u)=\langle u,v\rangle_V 
\]
for all $u,v\in V$.
Moreover, the {\em adjoint} of $A\in\Hom(V,W)$ is the (unique) linear operator $A^*\in\Hom(W,V)$ such that 
\[
\langle Av,w\rangle_W=\langle v,A^*w\rangle_V
\] 
for all $v\in V$ and $w\in W$. Observe that $(A^*)^* = A$. Also, the matrix $M_{A^*}$ associated with the adjoint operator $A^*$ equals the adjoint $(M_A)^* = (\overline{a}_{ik})_{\substack{i=1,2,\ldots,n\\k=1,2,\ldots,m}}$ of the matrix $M_A$ associated with $A$. 
Moreover, we say that $A$ is {\em unitary} if $A^*A=I_{V}$ and $AA^* = I_{W}$, where $I_{V} \in \Hom(V,V)$ denotes the identity map (note that a necessary condition for $A$ to be unitary is that $\dim(V) = \dim(W)$, i.e., $n=m$).

\begin{lemma}
Let $A\in\Hom(V,W)$. Then $A^T\theta_W=\theta_VA^*$.
\end{lemma}
\begin{proof}
For all $v\in V$ and $w\in W$ we have:
\[
(A^T\theta_W w)(v)=(\theta_W w)(Av)=\langle Av,w \rangle_W=\langle v,A^*w \rangle_V=(\theta_V A^*w)(v).
\]
\end{proof}

We now define the {\em conjugate} of $A\in\Hom(V,W)$ as the linear operator $\overline{A}=(A^*)^T\in\Hom(V',W')$.
Then, the associated matrix $M_{\overline{A}}$ equals the conjugate $\overline{M_{A}} = (\overline{a}_{ki})_{\substack{k=1,2,\ldots,m\\i=1,2,\ldots,n}}$ of the matrix associated with $A$.
Note that $\overline{\overline{A}}=A$ and $A^*=\overline{A}^T$ (here we implicitly use the canonical identification
of $V$ with its bidual $V''$).
As a consequence, $A$ is unitary if and only if $A^T\overline{A}=I_{V'}$ and $\overline{A}A^T = I_{W'}$.

Suppose now that $A\in\Hom(V',V)$. Then, again modulo the canonical identification of $V$ and $V''$, we have 
$$
A^T\in\Hom(V',V) \mbox{\ \ and \ \  } u'(A^Tv')=v'(Au')
$$ 
for all $u',v'\in V'$. We then say that $A\in\Hom(V',V)$ is {\em symmetric}
(resp. {\em antisymmetric} or {\em  skew-symmetric}) if $A^T=A$ (resp. $A^T=-A$). We denote by $\HomS(V',V)$ (resp. $\HomA(V',V)$) the space of all symmetric (resp. antisymmetric) operators in $\Hom(V',V)$. We have the elementary identity
\[
A=\frac{A+A^T}{2}+\frac{A-A^T}{2},
\]
where $\frac{A+A^T}{2}$ is symmetric and $\frac{A-A^T}{2}$ is antisymmetric: note that this is the unique decomposition of $A$ as a sum of a symmetric operator and an antisymmetric operator. This yields the direct sum decomposition
\begin{equation}\label{decsymanti}
\Hom(V',V)=\HomS(V',V)\oplus\HomA(V',V).
\end{equation}

Let now $A\in\Hom[(V\oplus W)', V\oplus W]$. Then there exist $A_1\in\Hom(V',V)$, $A_2\in\Hom(W',W)$, $A_3\in\Hom(W',V)$ and $A_4\in\Hom(V',W)$ such that 
\begin{equation}\label{tullio invidioso}
A(v'+w')=(A_1v'+A_3w')+(A_4v'+A_2w'),
\end{equation}
for all $v'\in V'$, $w'\in W'$. In other words, identifying $A$ with the operator matrix $\begin{pmatrix}
A_1&A_3\\
A_4&A_2
\end{pmatrix}$, we may express \eqref{tullio invidioso} in matrix form
\begin{equation}\label{matrixnotation}
\begin{pmatrix}
A_1&A_3\\
A_4&A_2
\end{pmatrix}
\begin{pmatrix}
v'\\
w'
\end{pmatrix}
=
\begin{pmatrix}
A_1v'+A_3w'\\
A_4v'+A_2w'
\end{pmatrix}.
\end{equation}
Since 
\[
\begin{pmatrix}
A_1&A_3\\
A_4&A_2
\end{pmatrix}^T
=
\begin{pmatrix}
A_1^T&A_4^T\\
A_3^T&A_2^T
\end{pmatrix}
\]
we have that $A=A^T$ if and only if $A_1=A_1^T$, $A_2=A_2^T$, $A_3=A_4^T$ and $A_4=A_3^T$ and this proves the first statement of the following lemma (the proof of the second statement is similar).

\begin{lemma}\label{lemmadirectsum}
\begin{enumerate}[{\rm (1)}]

\item
The map
\[
\begin{pmatrix}
A_1&A_3\\
A_3^T&A_2
\end{pmatrix}\longmapsto
(A_1,A_2,A_3^T)
\]
yields the isomorphism
\[
\HomS\left[(V\oplus W)',V\oplus W\right]\cong \HomS(V',V)\oplus\HomS(W',W)\oplus\Hom(V',W).
\]

\item
The map
\[
\begin{pmatrix}
A_1&A_3\\
-A_3^T&A_2
\end{pmatrix}\longmapsto
(A_1,A_2,A_3^T)
\]
yields the  isomorphism
\[
\HomA\left[(V\oplus W)',V\oplus W\right]\cong \HomA(V',V)\oplus\HomA(W',W)\oplus\Hom(V',W).
\]

\end{enumerate}
\end{lemma}

\subsection{Representation theory of finite groups}
We now recall some notions from the representation theory of finite groups. We refer to our monographs \cite{book,book2} for a complete exposition and detailed proofs. 

Let $G$ be a finite group. We always suppose that all $G$-representations $(\rho,W)$ are {\em unitary}: the representation space $W$ is finite dimensional hermitian and $\rho(g) \in \Hom(W,W)$ is unitary for every $g\in G$ (it is well known that every representation of a finite group over a complex vector space is unitarizable (cf. \cite[Proposition 3.3.1]{book})). We denote by $\widehat{G}$ a complete set of  pairwise inequivalent irreducible $G$-representations.

Given two $G$-representations $(\rho,W)$ and $(\sigma,V)$, we denote by $\Hom_G(W,V)$  (sometimes we shall also use the notation  $\Hom_G(\rho, \sigma)$) the space of all linear operators $A \colon W \to V$, called \emph{intertwiners} of $\rho$ and $\sigma$,  such that $A\rho(g) = \sigma(g)A$ for all $g \in G$.

Let $(\rho,W)$ be a $G$-representation. We denote by $\chi_\rho \colon G \to\C$ its character and, in the notation from Subsection \ref{ss:linear-algebra}, we denote by $(\rho',W')$ the {\em conjugate representation} defined by setting
\begin{equation}\label{defconjrep}
\rho'(g)=\rho(g^{-1})^T
\end{equation}
for all $g\in G$. We have $\chi_{\rho'}=\overline{\chi_\rho}$ (complex conjugation).
Suppose now that $(\rho,W)$ is irreducible. Then $\rho$ is said to be {\it complex} if $\chi_\rho \neq \chi_{\rho'}$, that is, $\rho$ and $\rho'$ are not (unitarily) equivalent; on the other hand, $\rho$ is called {\it self--conjugate} if $\chi_\rho = \chi_{\rho'}$, that is, $\rho$ and $\rho'$ are (unitarily) equivalent.
Clearly, $\rho$ is self--conjugate if and only if $\chi_\rho$ is real valued. 
The class of self--conjugate representations in turn may be splitted into two subclasses.  
Let $(\rho,W)$ be a self--conjugate $G$-representation and suppose that there exists an orthonormal basis
$\{w_1,w_2, \ldots, w_d\}$ in $W$ such that
the corresponding matrix coefficients are real valued: $u_{i,j}(g) =
\langle \rho(g)w_j, w_i\rangle \in \RR$ for all $g \in G$ and $i,j = 1,2,\ldots, d$.
Then $\rho$ is termed {\it real}. Otherwise, $\rho$ is said to be {\it quaternionic}.

\begin{lemma}\label{lemmaAAT}
Let $(\rho,W)$ be an irreducible, self--conjugate $G$-representation and let $A \in \Hom_G(W,W')$
be a unitary operator. Then, if $\rho$ is {\em real} one has $A\overline{A}=I_{W'}$ (equivalently, $A=A^T$), while if $\rho$ is {\em quaternionic} one has $A\overline{A}=-I_{W'}$ (equivalently, $A=-A^T$).
\end{lemma}
\begin{proof}
See \cite[Lemma 9.7.6]{book}
\end{proof}

Let  $n$ be a positive integer, and consider the diagonal subgroup $\widetilde{G}^n = \{g,g,\ldots,g): g \in G\}$ of $G^n = \underbrace{G \times G \times \ldots \times G}_{n \mbox{\ \tiny times}}$. Given $G$-representations $(\rho_i, V_i)$, $i=1,2,\ldots,n$, following our monograph,  we denote by $(\rho_1 \boxtimes \rho_2 \boxtimes \cdots \boxtimes \rho_n, V_1 \otimes V_2 \otimes \cdots \otimes V_n)$ their  \emph{external tensor product} which is a $G^n$-representation. Moreover we denote by $(\rho_1 \otimes \rho_2 \otimes
\cdots \otimes \rho_n, V_1 \otimes V_2 \otimes \cdots \otimes V_n)$ the {\it Kronecker product} of the $\rho_i$'s, that is the  $G$-representation defined by
$\rho_1 \otimes \rho_2 \otimes \cdots \otimes \rho_n = \Res^{G^n}_{\widetilde{G}^n} (\rho_1 \boxtimes \rho_2 \boxtimes
\cdots \boxtimes \rho_n)$.

\section{The $\tau$-Frobenius-Schur number}
\label{sec:tau}
In what follows,  $\tau \colon G \to G$ is  an {\em involutory anti-automorphism} of $G$, that is a bijection such that
\[
\tau(g_1g_2)=\tau(g_2)\tau(g_1) \qquad \text{ and }\qquad \tau^2(g)=g
\]
  for all $g_1,g_2,g \in G$. 
In particular, $\tau(1_G)=1_G$ and $\tau(g^{-1})=\tau(g)^{-1}$,  for all $g \in G$.
Let $(\rho,W)$ be a $G$-representation.
Then the associated \emph{$\tau$-conjugate representation} is the $G$-representation $(\rho^\tau,W')$ defined by setting
\begin{equation}\label{rhotau}
\rho^\tau(g)=\rho[\tau(g)]^T,
\end{equation}
that is 
\begin{equation}\label{tipico}
\left[\rho^\tau(g)w'\right](w)=w'[\rho(\tau(g))w],
\end{equation}  for all $g\in G$, $w'\in W'$ and $w\in W$. 
Note that if $\tau_{\text{\rm inv}} \colon G \to G$ is the involutory anti-automorphism of $G$ defined by $\tau_{\text{\rm inv}}(g) = g^{-1}$ for all $g \in G$, then $\rho^{\tau_{\text{\rm inv}}} = \rho'$ (cf. \eqref{defconjrep}).

\begin{remark} {\rm
Let $g_0 \in G$ and denote by $\tau_{g_0}$ the \emph{inner} involutory anti-automorphism of $G$ given by composing conjugation by $g_0$ and $\tau_{\text{\rm inv}}$, that is, $\tau_{g_0}(g) = g_0g^{-1}g_0^{-1}$ for all $g \in G$. 
Then, given a $G$-representation $(\rho,W)$ we have, for all $g\in G$, $w'\in W'$ and $w\in W$,
\[
\begin{split}
\left[\rho^{\tau_{g_0}}(g)w'\right](w) & = w'[\rho(\tau_{g_0}(g))w] = w'[\rho(g_0g^{-1}g_0^{-1})w] \\
& = w'[\rho(g_0)\rho(g^{-1})\rho(g_0)^{-1}w] \\
& = \left[(\rho(g_0)\rho(g^{-1})\rho(g_0)^{-1})^Tw'\right](w) \\
& = \left[\rho(g_0^{-1})^T\rho^{\tau_{\text{\rm inv}}}(g)\rho(g_0)^Tw'\right](w)
\end{split}
\]
yielding $\rho^{\tau_{g_0}}(g) =(\rho(g_0)^T)^{-1}\rho^{\tau_{\text{\rm inv}}}(g)\rho(g_0)^T$, so that
\[
\rho^{\tau_{g_0}} \sim \rho^{\tau_{\text{\rm inv}}}.
\]
}
\end{remark}

\begin{proposition}
\label{p:3.1}
\begin{enumerate}[{\rm (1)}]
\item
The $G$-representation $(\rho^\tau,W')$ is irreducible if and only if $(\rho,W)$ is irreducible.
\item
If $A\in\Hom_G(\rho^\tau,\rho)$ then also $A^T\in\Hom_G(\rho^\tau,\rho)$ and we have the direct sum decomposition
\[
\Hom_G(\rho^\tau,\rho)=\HomS_G(\rho^\tau,\rho)\oplus\HomA_G(\rho^\tau,\rho),
\]
where $\HomS_G = \HomS \cap \Hom_G$ and $\HomA_G = \HomA \cap \Hom_G$ (compare with \eqref{decsymanti}).
\end{enumerate}
\end{proposition}

\begin{proof}
(1) Suppose first that $\rho$ is reducible and let $U\leq W$ be a nontrivial $\rho$-invariant subspace. 
Then $Z=\{w'\in W': w'(u)=0 \text{ for all } u\in U\}\leq W'$ is nontrivial and $\rho^\tau$-invariant, thus showing
that $\rho^\tau$ is also reducible. Since $(\rho^\tau)^\tau = \rho$, by applying the previous argument we also
deduce the converse.

(2) Let $A\in\Hom_G(\rho^\tau,\rho)$ and $g \in G$. Then, by transposing  the identity $A\rho^\tau(g)=\rho(g)A$ we get $\rho^\tau(g)^TA^T=A^T\rho(g)^T$ which, by \eqref{rhotau}, becomes $\rho[\tau(g)]A^T=A^T\rho^\tau[\tau(g)]$.
Since $\tau$ is bijective, by replacing $\tau(g)$ with $g$, we finally obtain $\rho(g)A^T=A^T\rho^\tau(g)$, thus showing that $A^T\in\Hom_G(\rho^\tau,\rho)$. The direct sum decomposition is obvious.
\end{proof}

\begin{lemma}
\label{l:2}
Let $(\rho,W)$ and $(\sigma,V)$ be two $G$-representations. Then the following isomorphisms hold:
\begin{equation}\label{sigmataurho}
\Hom_G(\sigma^\tau,\rho)\cong\Hom_G(\rho^\tau,\sigma),
\end{equation}
\begin{equation}\label{sigmataurhoS}
\HomS_G\left[(\sigma\oplus\rho)^\tau,\sigma\oplus\rho\right]\cong\HomS_G(\sigma^\tau,\sigma)\oplus\HomS_G(\rho^\tau,\rho)\oplus\Hom_G(\sigma^\tau,\rho)
\end{equation}
and
\begin{equation}\label{sigmataurhoA}
\HomA_G\left[(\sigma\oplus\rho)^\tau,\sigma\oplus\rho\right]\cong\HomA_G(\sigma^\tau,\sigma)\oplus\HomA_G(\rho^\tau,\rho)\oplus\Hom_G(\sigma^\tau,\rho).
\end{equation}
\end{lemma}
\begin{proof}
The isomorphism \eqref{sigmataurho} is realized by the map $A\mapsto A^T$. The isomorhism \eqref{sigmataurhoS} (resp. \eqref{sigmataurhoA}) is realized by the map in Lemma \eqref{lemmadirectsum}.(1) (resp. Lemma \eqref{lemmadirectsum}.(2)), keeping into account that in the matrix notation \eqref{matrixnotation} $A$ is an intertwining operator if and only if $A_1, A_2, A_3, A_4$ are intertwining operators.
\end{proof}

\begin{definition}
\label{rhotau2}
{\rm
The $\tau$-{\em Frobenius}-{\em Schur number} of a $G$ representation $(\rho,W)$ is the integer number $C_\tau(\rho)$
defined by
\[
C_\tau(\rho)=\dim\HomS_G(\rho^\tau,\rho)-\dim\HomA_G(\rho^\tau,\rho).
\]}
\end{definition}

We also set $C(\rho)=\dim\HomS_G(\rho',\rho)-\dim\HomA_G(\rho',\rho)$, that is $C(\rho) = C_{\tau_{\text{\rm inv}}}(\rho)$. 
We start by examining $C_\tau(\rho)$ and $C(\rho)$ when $\rho$ is irreducible.

\begin{theorem}
\label{thmCtC}
Suppose that $\rho$ is irreducible. Then
\begin{enumerate}[{\rm (1)}]
\item
$C_\tau(\rho)\in\{-1,0,1\}$. Moreover, $C_\tau(\rho) = 0$ (resp. $C_\tau(\rho) = \pm 1$) if and only if
$\rho^\tau \not\sim \rho$ (resp. $\rho^\tau \sim \rho$).

\item In particular,
\label{Crho}
 \[
C(\rho)=\begin{cases}
1&\text{ if }\rho\text{ is real}\\
0&\text{ if }\rho\text{ is complex}\\
-1&\text{ if }\rho\text{ is quaternionic.}
\end{cases}
\]
\end{enumerate}
\end{theorem}
\begin{proof}
(1) If $\rho^\tau\not\sim\rho$ then $\dim\Hom_G(\rho^\tau,\rho)=0$ and therefore $C_\tau(\rho)=0$. Now suppose that $\rho\sim\rho^\tau$. If $A\in \Hom_G(\rho^\tau,\rho)$, $A\neq 0$, then also $A^T\in \Hom_G(\rho^\tau,\rho)$ and therefore, by Proposition \ref{p:3.1} and Schur's lemma, there exists $\lambda\in\CC$ such that $A^T=\lambda A$. By transposing we get $A=\lambda A^T=\lambda^2A$, which implies that $\lambda=\pm1$. If $\lambda=1$ then $A$ is symmetric and $C_\tau(\rho)=\dim\HomS_G(\rho^\tau,\rho)-\dim\HomA_G(\rho^\tau,\rho)=1-0=1$; similarly, if $\lambda =-1$ then $C_\tau(\rho)=-1$.

(2) If $\rho$ is complex then $\rho'\not\sim\rho$ and therefore $\dim\Hom_G(\rho',\rho)=0$ and $C(\rho)=0$. If $\rho$ is self--adjoint then $\dim\Hom_G(\rho',\rho)=1$ and this space is spanned by a unitary matrix $A$ as in Lemma \ref{lemmaAAT}, which is symmetric if $\rho$ is real and antisymmetric if $\rho$ is quaternionic.
\end{proof}

We now examine the behaviour of $C_\tau$ with respect to direct sums and tensor products.

\begin{proposition}
\label{PropCtausum}
Let $(\rho,W)$ and $(\sigma,V)$ be two $G$-representations. Then
\[
C_\tau(\sigma\oplus\rho)=C_\tau(\sigma)+C_\tau(\rho).
\]
\end{proposition}
\begin{proof}
It is an immediate consequence of \eqref{sigmataurhoS} and \eqref{sigmataurhoA}.
\end{proof}

\begin{proposition}
\label{p:5}
Suppose that $G=G_1\times G_2$ and that $\tau$ satisfies
$\tau(G_1\times\{1_{G_2}\})=G_1\times\{1_{G_2}\}$ and $\tau(\{1_{G_1}\}\times G_2)=\{1_{G_1}\}\times G_2$.  Let $(\rho_i,W_i)$ be a $G_i$-representation for $i=1,2$. Then 
\begin{equation}\label{Ctauprod}
C_\tau(\rho_1\boxtimes\rho_2)=C_\tau(\rho_1)C_\tau(\rho_2).
\end{equation}
\end{proposition}
\begin{proof}
We first prove \eqref{Ctauprod} under the assumption that both $\rho_1$ and $\rho_2$ are irreducible. The representation $\left(\rho_1\boxtimes\rho_2\right)^\tau\sim\rho_1^\tau\boxtimes\rho_2^\tau$ is equivalent to $\rho_1\boxtimes\rho_2$ if and only if $\rho_1\sim\rho_1^\tau$ and $\rho_2\sim\rho_2^\tau$. Therefore $C_\tau(\rho_1\boxtimes\rho_2)=0$ if and only if $C_\tau(\rho_1)=0$ or $C_\tau(\rho_2)=0$. On the other hand, if $\rho_1\sim\rho_1^\tau$, $\rho_2\sim\rho_2^\tau$ and $A_i$ spans $\Hom_{G_i}\left(\rho_i^\tau,\rho_i\right)$, for $i=1,2$, then $\Hom_G\left[(\rho_1\boxtimes\rho_2)^\tau,\rho_1\boxtimes\rho_2\right]$ is spanned by $A_1\otimes A_2$. It is easy to check that $(A_1\otimes A_2)^T=A_1^T\otimes A_2^T$ so that $A_1\otimes A_2$ is symmetric if and only if $A_1$ and $A_2$ are both symmetric or antisimmetric, while $A_1\otimes A_2$ is antisymmetric if and only if one of the operators $A_1$ and $A_2$ is symmetric and the other is antisymmetric. In both cases, \eqref{Ctauprod} follows.

Now we remove the irreducibility assumption and we suppose that 
\[
\rho_1=\bigoplus_{i=1}^nm_i\sigma_i\qquad\text{ and }\qquad \rho_2=\bigoplus_{j=1}^kh_j\theta_j
\]
are the decompositions of $\rho_1$ and $\rho_2$ into irreducible representations. Then
\[
\begin{split}
C_\tau(\rho_1\boxtimes\rho_2)=& C_\tau\left[\bigoplus_{i=1}^n\bigoplus_{j=1}^km_ih_j(\sigma_i\boxtimes\theta_j)\right]\\
(\text{by Proposition }\ref{PropCtausum})\ =& \sum_{i=1}^n\sum_{j=1}^km_ih_jC_\tau(\sigma_i\boxtimes\theta_j)\\
(\text{by the first part of the proof})\ =& \sum_{i=1}^n\sum_{j=1}^km_ih_jC_\tau(\sigma_i)C_\tau(\theta_j)\\
=& \left[\sum_{i=1}^nm_iC_\tau(\sigma_i)\right]\left[\sum_{j=1}^kh_jC_\tau(\theta_j)\right]\\
(\text{again by Proposition }\ref{PropCtausum})\ \ =& C_\tau(\rho_1)C_\tau(\rho_2).
\end{split}
\]
\end{proof}

Note that when $\tau = \tau_{\text{\rm inv}}$ the first part of the proof of the preceding proposition may be also deduced from Theorem \ref{thmCtC}.\eqref{Crho}.

Let now $\omega \colon G \to G$ be  another involutory anti-automorphism of $G$ commuting with $\tau$: $\omega \tau = \tau \omega$.
Clearly, the composition $\omega\tau$ is now an (involutory) automorphism of $G$. Moreover we have
\[
(\rho^\tau)^\omega(g) = \left(\rho^\tau(\omega(g)\right)^T = \rho(\tau(\omega(g))) = \rho(\omega(\tau(g))) = (\rho^\omega)^\tau(g)
\]
for all $g \in G$, that is, 
\begin{equation}
\label{e:tau-omega}
(\rho^\tau)^\omega = (\rho^\omega)^\tau.
\end{equation}

\begin{lemma}
\label{l:7}
Let $\omega, \tau$ and $(\rho,W)$ (not necessarily irreducible) be as above. Then
\[
C_\tau(\rho^\omega) = C_\tau(\rho).
\]
\end{lemma}
\begin{proof}
By virtue of Proposition \ref{PropCtausum} it suffices to examine the case when $\rho$ is irreducible.
If $C_\tau(\rho)=0$ then $\rho^\tau \not\sim \rho$ and therefore $(\rho^\omega)^\tau = (\rho^\tau)^\omega
\not\sim \rho^\omega$ (recall that $\omega$ is involutory). We deduce that $C_\tau(\rho^\omega)=0$ as well.

Suppose now that $\rho^\tau \sim \rho$ and let $A \in \Hom_G(\rho^\tau,\rho)$ be a nontrivial unitary intertwiner. Then,
for all $g \in G$ we have $\rho(g)A = A\rho^\tau(g)$ so that
\[
(\rho^\tau)^\omega(g)A^T = \left(\rho^\tau(\omega(g))\right)^TA^T = \left(A \rho^\tau(\omega(g))\right)^T = 
\left(\rho(\omega(g))A\right)^T = A^T \rho^\omega(g).
\]
This shows that $\Hom_G(\rho^\omega, (\rho^\omega)^\tau) \equiv \Hom_G(\rho^\omega, (\rho^\tau)^\omega)$ is
spanned by $A^T$, so that $\Hom_G((\rho^\omega)^\tau, \rho^\omega)$ is spanned by $(A^T)^{-1} \equiv (A^T)^* = \overline{A}$.
Thus since $\overline{A}$ is symmetric (resp. antisymmetric) if and only if $A$ is symmetric (resp. antisymmetric), we
deduce that $C_\tau(\rho^\omega) = 1$ (resp. $C_\tau(\rho^\omega) = -1$) if and only if $C_\tau(\rho) = 1$ (resp. $C_\tau(\rho) = -1$).
\end{proof}

By taking $\omega = \tau$ we deduce the following
\begin{corollary}\label{ppii}
$C_\tau(\rho^\tau) = C_\tau(\rho)$.
\qed
\end{corollary}

Let now $K \leq G$ be  a $\tau$-invariant (that is $\tau(K) = K$) subgroup. It is clear that if $(\rho,W)$ is
a $G$-representation, then $\Res^G_K(\rho^\tau) = (\Res^G_K \rho)^\tau$. Conversely, suppose now that
$(\sigma,V)$ is a $K$-representation and let us show that $\Ind^G_K(\sigma^\tau) \sim (\Ind^G_K \sigma)^\tau$.
We set $\Ind^G_K V = \{F\in V^G: F(gk) = \sigma(k^{-1})F(g), \mbox{ for all } g\in G, k \in K\}$, $\rho =
\Ind^G_K \sigma$, so that (cf. \cite[Definition 1.6.1]{book2}) $[\rho(g_0)F](g) = F(g_0^{-1}g)$ for all
$g, g_0 \in G$ and $F \in \Ind_K^GV$, and 
$\Ind^G_K V' = \{F'\in (V')^G: F'(gk) = \sigma^\tau(k^{-1})F'(g), \mbox{ for all } g\in G, k \in K\}$,
$\widetilde{\rho} = \Ind^G_K \sigma^\tau$, so that $[\widetilde{\rho}(g_0)F'](g) = F'(g_0^{-1}g)$ for all
$g, g_0 \in G$ and $F' \in  \Ind_K^GV'$.

\begin{lemma}
\label{l:9}
The linear map $\xi\colon \Ind^G_K V' \to (\Ind^G_K V)'$ defined by setting
\[
(\xi F')(F) = \frac{1}{|K|} \sum_{g \in G} F'(g)(F(\tau(g^{-1})))
\]
for all $F \in \Ind^G_K V$, $F' \in \Ind^G_K V'$, yields an isomorphism between $\Ind^G_K \sigma^\tau$ and
$(\Ind^G_K \sigma)^\tau$.
\end{lemma}
\begin{proof}
Let $S \subseteq G$ be a complete system of representatives for the set $G/K$ of left cosets of $K$ in $G$, so that
$G = \coprod_{s \in S}sK$. On the one hand: 
\[
\begin{split}
(\xi F')(F) & = \frac{1}{|K|} \sum_{g \in G} F'(g)(F(\tau(g^{-1})))\\
& = \frac{1}{|K|} \sum_{s \in S}\sum_{k \in K} F'(sk)(F(\tau((sk)^{-1})))\\
& = \frac{1}{|K|} \sum_{s \in S}\sum_{k \in K} [\sigma^\tau(k^{-1})F'(s)]\left(\sigma(\tau(k))F(\tau(s)^{-1})\right)\\
(\mbox{by \eqref{tipico}}) \ \ & = \frac{1}{|K|} \sum_{s \in S}\sum_{k \in K} F'(s)\left(\sigma(\tau(k^{-1}))\sigma(\tau(k)) F(\tau(s)^{-1})\right)\\
& = \sum_{s \in S} F'(s)\left(F(\tau(s)^{-1})\right).
\end{split}
\]
Since $F'$ is uniquely determined by $(F'(s))_{s \in S}$, we deduce that $\xi$ is injective. Moreover, as
$\dim \Ind^G_K V' = \dim (\Ind^G_K V)'$ we deduce that $\xi$ is indeed bijective.
On the other hand, for $g_0 \in G$, we have
\[
\begin{split}
[\rho^\tau(g_0)\xi F'](F) & =  [\xi F'](\rho(\tau(g_0))F) \ \ \ (\mbox{by \eqref{tipico}}\\
& = \frac{1}{|K|} \sum_{g \in G} F'(g)\left(F(\tau(g_0)^{-1}\tau(g)^{-1})\right)\\
\mbox{(setting $g_1 = g_0g$)} \ & = \frac{1}{|K|} \sum_{g_1 \in G} F'(g_0^{-1}g_1)\left(F(\tau(g_1)^{-1})\right)\\
 & = \frac{1}{|K|} \sum_{g_1 \in G} \left\{[\widetilde{\rho}(g_0)F'](g_1)\right\}\left(F(\tau(g_1)^{-1})\right)\\
 & = \left\{\xi[\widetilde{\rho}(g_0)F']\right\}(F).
 \end{split}
 \]
This shows that  $\rho^\tau(g_0)\xi = \xi \widetilde{\rho}(g_0)$ for all $g_0 \in G$ so that $\xi \in \Hom_G\left(\Ind^G_K \sigma^\tau,(\Ind^G_K \sigma)^\tau\right)$, completing the proof.
\end{proof}

\begin{theorem}
\label{t:10}
Let $K \leq G$ be a $\tau$-invariant subgroup. Let also $(\rho,W)$ be an irreducible $G$-representation
whose restriction $\Res^G_K\rho$ is \emph{multiplicity-free}, that is, $\Res^G_K (\rho,W) = \bigoplus_{i=1}^m (\sigma_i,V_i)$, with $\sigma_i$ irreducible and $\sigma_i \not\sim \sigma_j$ for $1 \leq i \neq j \leq m$.
Suppose that $\rho^\tau \sim \rho$ and $\sigma_i^\tau \sim \sigma_i$ for $i=1,2,\ldots,m$.
Then
\[
C_\tau(\sigma_i) = C_\tau(\rho)
\]
for all $i=1,2,\ldots,m$.
\end{theorem}
\begin{proof}
Let us set, for $i=1,2,\ldots,m$, 
$$
W_i' = \left\{w' \in W': \ker w' = \left(\bigoplus_{j=1}^{i-1} V_j\right) \bigoplus \left(\bigoplus_{j=i+1}^{m} V_j\right)\right\} \cong V_i'.
$$
If we identify $V'_i$ with $W_i'$  then $(\Res^G_K \rho^\tau)\vert_{W_i'} = \sigma_i^\tau$: indeed
\[
[\rho^\tau(k)w'](v) = w'[\rho(\tau(k))v] = w'[\sigma_i(\tau(k))v] = [\sigma_i^\tau(k)w'](v),
\]
for all $w \in W_i'$, $v \in V_i$, and $k \in K$ (clearly $W_i'$ is $K$-invariant).

Now, if $A \in \Hom_G(\rho^\tau, \rho)$, that is, $A\rho^\tau(g) = \rho(g)A$ for all $g \in G$, we deduce that
\[
A \sigma_i^\tau(k)w' = A \rho^\tau(k)w' = \rho(k)Aw'
\]
for all $w' \in W_i'$ and $k \in K$. It follows that $A \vert_{W_i'} \colon W_i' \to V_i$ (recall that
$\sigma_i^\tau \sim \sigma_i$ and $\sigma_i^\tau \not\sim \sigma_j$ for $1 \leq i \neq j \leq m$).
Thus, setting $A_i = A \vert_{W_i'}$ we have $A = A_1 \oplus A_2 \oplus \cdots \oplus A_m$,
$A^T = A_1^T \oplus A_2^T \oplus \cdots \oplus A_m^T$, and $A$ is symmetric (resp. antisymmetric) if and
only if $A_1, A_2, \ldots, A_m$ are all symmetric (resp. antisymmetric).
\end{proof}

\begin{lemma}
\label{l:11}
Let $\rho$ be a $G$-representation and denote by $\rho = \bigoplus_{j=1}^n \rho_j$ a decomposition into
irreducibles (now the sub-representations $\rho_j$ need not be pairwise inequivalent). Then
$\dim \HomA_G(\rho^\tau, \rho) = 0$ if and only if for every $j=1,2,\ldots, n$ one of the following
conditions holds:
\begin{enumerate}[{\rm (i)}]
\item $C_\tau(\rho_j) = 1$ and $\rho_j \not\sim \rho_k$ for all $k \neq j$;
\item $C_\tau(\rho_j) = 0$ and $\rho_j^\tau \not\sim \rho_k$ for all $k \neq j$.
\end{enumerate}
Different $j$'s may satisfy different conditions.
\end{lemma}
\begin{proof}
Let us set $\sigma_i = \bigoplus_{j=i}^n \rho_j$ for $i=1,2,\ldots,n$. By repeatedly applying Lemma \ref{l:2}, 
we deduce that
\[
\begin{split}
\dim \HomA_G(\rho^\tau, \rho) & = \dim \HomA_G((\rho_1 \oplus \sigma_2)^\tau,\rho_1 \oplus \sigma_2)\\
& = \dim \HomA_G(\rho_1^\tau,\rho_1) + \dim \HomA_G(\sigma_2^\tau,\sigma_2) + \dim \Hom_G(\rho_1^\tau,\sigma_2)\\
&  \cdots \\
& = \sum_{j=1}^n \dim \HomA_G(\rho_j^\tau,\rho_j) + \sum_{j=1}^{n-1} \dim \Hom_G(\rho_j^\tau,\sigma_{j+1}).
\end{split}
\]
Thus $\dim \HomA_G(\rho^\tau, \rho) = 0$ if and only if $\dim \HomA_G(\rho_j^\tau,\rho_j) = 0$  for all $j=1,2,\ldots,n$ and $\dim \Hom_G(\rho_j^\tau,\sigma_{j+1}) = 0$ for all $j=1,2,\ldots,n-1$.
It follows that if $\rho_j^\tau \sim \rho_j$ we necessarily have $C_\tau(\rho_j) = 1$ and $\rho_j \not\sim \rho_k$
for all $j < k \leq n$, while if $\rho_j^\tau \not\sim \rho_j$ we necessarily have $C_\tau(\rho_j) = 0$ and $\rho_j^\tau \not\sim \rho_k$ for all $j < k \leq n$. Now, in both cases, the condition $k > j$ can be replaced by
$k \neq j$: since the order  in $\rho = \bigoplus_{j = 1}^n \rho_j$ is arbitrary, we may always suppose $j = 1$.
\end{proof}

\section{Multiplicity-free permutation representations: the Mackey-Gelfand criterion}
\label{s:MFPR}
Let $G$ be a finite group and suppose we are given  a transitive action $\pi \colon G \to \Sym(X)$ of $G$ on a (finite) set $X$. Fix $x_0 \in X$ and denote by
$K = \Stab_G^\pi(x_0) = \{g \in G: \pi(g)x_0 = x_0\} \leq G$ its $G$-{\it stabilizer}. Then we identify the \emph{homogeneous space} $X$ with the set $G/K$ of left cosets of $K$ in $G$. 
This way, the action is given by $\pi(g)x = (gg')K$ for all $g \in G$ and $x = g'K \in X$ (note that, in
particular, $x_0 \equiv K$). We also denote by ${\mathcal O}_G^\pi(X)$ the corresponding set of $G$-orbits in $X$.

Let also $\tau \colon G \to G$ be an involutory anti-automorphism of $G$ which does not necessarily preserve $K$. 
Let $Y = G/\tau(K)$ denote the corresponding homogeneous space and by $y_0 = \tau(K) \in Y$ the corresponding $\tau(K)$-fixed point.

We denote by $L(X)$ the vector space of all functions $f \colon X \to \C$ and denote by
$(\lambda_\pi,L(X))$ the \emph{permutation representation} associated with the action $\pi$, that is, the $G$-representation defined by
\[
[\lambda_\pi(g)f](x) = f(\pi(g^{-1})x)
\]
for all $g \in G$, $f \in L(X)$ and $x \in X$.
Also, we denote by $\gamma \colon G \to \Sym(Y)$ the action of $G$ on $Y$:  $\gamma(g)y =
gg'\tau(K)$ for all $g \in G$ and $y = g'\tau(K) \in Y$. We then define a map 
$\theta \colon X \to Y$
by setting, for every $x \in X$, 
\begin{equation}
\label{e:theta}
\theta(x) = \gamma(\tau(g^{-1}))y_0 \quad \mbox{where $g \in G$ satisfies $\pi(g)x_0 = x$.}
\end{equation}
 Note that the map is well defined:  if 
$g_1, g_2 \in G$ satisfy $\pi(g_1)x_0 = \pi(g_2)x_0$, then there exists $k \in K$ such that $g_2 = g_1k$ and therefore \[
\gamma(\tau(g_2^{-1}))y_0 = \gamma(\tau(k^{-1}g_1^{-1}))y_0 = \gamma(\tau(g_1^{-1})) \gamma(\tau(k^{-1}))y_0 = \gamma(\tau(g_1^{-1}))y_0.
\] 
It is clear that $\theta$ is a bijection and that $\theta(x_0) = y_0$.

We now define a second action $\pi^\tau \colon G \to \Sym(X)$ by setting
\[
\pi^\tau(g)x = \pi(\tau(g^{-1}))x
\]
for all $g \in G$ and $x\in X$.
The associated permutation representation $(\lambda_{\pi^\tau},L(X))$  is then given by
\[
[\lambda_{\pi^\tau}(g)f](x) = f(\pi(\tau(g))x)
\]
for all $g \in G$, $f \in L(X)$ and $x \in X$.

The  $\tau$-conjugate representation (cf. Section \ref{sec:tau})
$(\lambda_\pi^\tau, L(X)')$   of $\lambda_\pi$ is then given by
$[\lambda_\pi^\tau(g)\varphi'](f) = \varphi'(\lambda_\pi(\tau(g))f)$ for all
$g \in G$, $\varphi' \in L(X)'$ and $f \in L(X)$.
In the following, we identify the dual $L(X)'$ with $L(X)$ via the bijective linear map
\begin{equation}
\label{duale}
\begin{array}{ccc}
L(X) &  \to & L(X)'\\
\varphi & \mapsto & \varphi'
\end{array}
\end{equation}
where
\[
\varphi'(f) = \sum_{x \in X} \varphi(x)f(x)
\]
for all $f \in L(X)$.

Finally, we fix $S \subseteq G$ a complete set of representatives of the double $\tau(K)\backslash G/K$-cosets, so that
\[
G = \coprod_{s \in S} \tau(K)sK.
\]
Observe that $\tau(\tau(K)sK) = \tau(K)\tau(s)K$ and therefore $\tau(S)$  is also  a system of representatives of the double cosets.
Let also $\pi^\tau \times \pi \colon G \to \Sym(X\times X)$  be  the action defined by 
\[
(\pi^\tau \times \pi)(g)(x_1,x_2) = (\pi^\tau(g)x_1, \pi(g)x_2)
\]
for all $g \in G$ and $x_1, x_2 \in X$.
Also, we denote by $\flat \in \Sym(X \times X)$ the involution defined by $(x_1,x_2)^\flat = (x_2, x_1)$
for all $x_1, x_2 \in X$. 

\begin{lemma}
\label{l:1MF}
\begin{enumerate}[{\rm (1)}]
\item 
The bijective map $\theta$ in \eqref{e:theta} satisfies  $\theta \pi^\tau(g) = \gamma(g)\theta$ for all $g \in G$.
Thus $\Stab_{G}^{\pi^\tau}(x_0) = \tau(K)$.
\item 
We have 
\begin{equation}
\label{auroramia}
(\lambda_\pi^\tau, L(X)') \sim (\lambda_{\pi^\tau},L(X))
\end{equation}
via the bijective map \eqref{duale}.
\item
The maps
\[
\begin{array}{cccc}
\Psi_S \colon & S & \longrightarrow & {\mathcal O}_{\tau(K)}^\pi(X) \equiv {\mathcal O}_{K}^{\pi^\tau}(X)\\
& s & \mapsto & \{\pi(\tau(k)s)x_0: k \in K\} 
\end{array}
\]
and 
\[
\begin{array}{cccc}
\Xi_S \colon & S &  \longrightarrow & {\mathcal O}_{G}^{\pi^\tau \times \pi}(X \times X)\\
& s &  \mapsto & \{(\pi^\tau(g)x_0, \pi(gs)x_0): g \in G\} 
\end{array}
\]
are bijective. Moreover, $\Xi_{\tau(S)}(\tau(s)) = \left(\Xi_S(s)\right)^\flat$ for all $s \in S$,
where  $\Xi_{\tau(S)}(\tau(s)) = \{(\pi^\tau(g)x_0, \pi(g\tau(s))x_0): g \in G\}$.
\end{enumerate}
\end{lemma}
\begin{proof} (1) Let $x\in X$ and $g_0 \in G$ be such that $x = \pi(g_0)x_0$. Then for all $g \in G$ we have
\[
\begin{split}
\theta(\pi^\tau(g)x) & = \theta(\pi(\tau(g^{-1}))\pi(g_0)x_0)\\
& = \theta(\pi(\tau(g^{-1})g_0)x_0)\\
  \mbox{(by \eqref{e:theta})} \  & = \gamma(g \tau(g_0^{-1}))y_0\\
\mbox{(since $\gamma$ is an action)} \ & = \gamma(g)\gamma(\tau(g_0^{-1}))y_0\\
\mbox{(again by \eqref{e:theta})} \ & = \gamma(g) \theta(\pi(g_0)x_0)\\
& = \gamma(g) \theta(x).
\end{split}
\]

(2) Let now $\varphi, f \in L(X)$ and $g \in G$. Then we have
\[
\begin{split}
[\lambda_\pi^\tau(g) \varphi'](f) & = \varphi'\left(\lambda_\pi(\tau(g))f\right)\\
& = \sum_{x \in X} \varphi(x) f(\pi(\tau(g^{-1}))x)\\
\mbox{(setting $z=\pi(\tau(g^{-1}))x$)} \ & = \sum_{z \in X} \varphi(\pi(\tau(g))z) f(z)\\
& = \sum_{z \in X} [\lambda_{\pi^\tau}(g)\varphi](z)f(z)\\
& = \left(\lambda_{\pi^\tau}(g)\varphi\right)'(f),
\end{split}
\]
so that $\lambda_\pi^\tau(g) \varphi' = \left(\lambda_{\pi^\tau}(g)\varphi\right)'$. In other words,
the bijective map \eqref{duale} yields the equivalence \eqref{auroramia}.

(3) By definition of $S$,  $\Psi_S$ is well defined and bijective. As for $\Xi_S$, let $x_1, x_2 \in X$.
Since the action $\pi^\tau$ (resp. $\pi$) is transitive, we can find $g_1 \in G$ ($g_2 \in G$) such that 
$\pi^\tau(g_1)x_0 = x_1$ (resp. $\pi(g_2)x_0 = x_2$). Let $k_1, k_2 \in K$ and $s \in S$ be such that
$\tau(k_1)sk_2 = g_1^{-1}g_2$. Setting $g=g_1\tau(k_1)$, we then have
\[
(\pi^\tau(g)x_0, \pi(gs)x_0)  = (\pi^\tau(g_1)\pi(k_1^{-1})x_0, \pi(g_1)\pi(\tau(k_1)sk_2)x_0) = (x_1,x_2).
\]
This shows that $\Xi_S$ is surjective. Now we show that it is injective: if $\Xi_S(s_1) = \Xi_S(s_2)$ for $s_1, s_2 \in S$ then there exists $g \in G$ such that $(\pi^\tau(g)x_0, \pi(gs_2)x_0) = (x_0, \pi(s_1)x_0)$, and  this implies that $g \in \tau(K)$ and 
$s_1^{-1}gs_2\in K$, so that  $s_1 \in \tau(K)s_2K$ and necessarily $s_1 = s_2$
Finally, let $s \in S$ and $g \in G$ and set $g' = \tau(g^{-1})s^{-1}$. It is then immediate to check that
\[
(\pi^\tau(g)x_0, \pi(g\tau(s))x_0) = (\pi(g's)x_0, \pi^\tau(g')x_0) = (\pi^\tau(g')x_0,\pi(g's)x_0)^\flat.
\]
\end{proof}

\begin{corollary}
\label{c:auroretta}
Suppose that $K = \tau(K)$. Then $\pi \equiv \gamma$, so that $\theta\pi^\tau(g) = \pi(g)\theta$ for all $g \in G$ and 
$\lambda_\pi^\tau \sim \lambda_\pi$.
\end{corollary}
\begin{proof}
Just note that defining  $T\colon L(X) \to L(X)$ by setting $(Tf)(x) = f(\theta(x))$ for all $f \in L(X)$,  $x\in X$ , we get a linear bijection such
that $T\lambda_\pi(g) = \lambda_{\pi^\tau}(g)T$, for all $g \in G$.
\end{proof}

\begin{remark}\label{remup}
{\rm  Lemma \ref{l:1MF} generalizes the  well known facts that the maps $s \mapsto \{\pi(ks)x_0, k \in K\}$ and 
$s \mapsto \{(\pi(g)x_0, \pi(gs)x_0): g \in G\}$ are bijections respectively between $S$ and $\mathcal{O}_K^\pi(X)$ and between $S$ and $\mathcal{O}_G^{\pi\times \pi}(X \times X)$. See  \cite[Section 3.13 ]{book} and 
 \cite[Section 1.5.3]{book2}.}
\end{remark} 

We shall say that a $(\pi^\tau \times \pi)$-orbit of $G$ on $X \times X$ is $\tau$-\emph{symmetric} (resp. $\tau$-\emph{antisymmetric}) provided it is (resp. is not) invariant under the flip $\flat \colon (x_1,x_2) \mapsto (x_2,x_1)$.  We then denote by $m_1$ (resp. $m_2$) the number of such $\tau$-symmetric (resp. $\tau$-antisymmetric) orbits. From Lemma \ref{l:1MF}.(3) we then have $m_1$ (resp. $m_2$) equals the number of $s \in S$ such that $\tau(s) \in \tau(K) sK$ (resp. $\tau(s) \not\in \tau(K)sK$). Note that $m_1 + m_2 = |S|$ and that $m_2$ is even.

\begin{theorem} We have
\label{t:2MF}
\begin{equation}
\label{e:laforgia1}
\dim \HomS_G(\lambda_\pi^\tau, \lambda_\pi) = m_1 + \frac{1}{2}m_2;
\end{equation}
\begin{equation}
\label{e:laforgia2}
\dim \HomA_G(\lambda_\pi^\tau, \lambda_\pi) = \frac{1}{2}m_2.
\end{equation}
\end{theorem}
\begin{proof}
Denoting, as in \eqref{duale} , by $\varphi \mapsto \varphi'$ the
identification of $L(X)$ and its dual $L(X)'$ and recalling that $\lambda^\tau_\pi\sim \lambda_{\pi^\tau}$  (cf. \eqref{auroramia}), with every $A \in \Hom_G(\lambda_\pi^\tau, \lambda_\pi) \cong \Hom_G(\lambda_{\pi^\tau}, \lambda_\pi)$ we associate
a complex matrix $(a(x_1,x_2))_{x_1,x_2 \in X}$ such that
\[
[A\varphi](x_2) = \sum_{x_1 \in X} a(x_1,x_2)\varphi(x_1)
\]
for all $x_1 \in X$.  Then $A$ is an intertwiner if and only if 
\[
a(\pi^{\tau}(g)x_1,\pi(g)x_2) = a(x_1,x_2)
\]
for all $g \in G$ and $x_1, x_2 \in X$, that is, if and only if  $a(x_1,x_2)$ is constant on the $(\pi^\tau \times \pi)$-orbits of $G$ on
$X \times X$.
Now, if $A$ is symmetric, $a(x_1,x_2)$ must assume the same values on coupled antisymmetric orbits and therefore
\eqref{e:laforgia1} follows.
On the other hand, if $A$ is antisymmetric, $a(x_1,x_2)$ must vanish on all symmetric orbits and assume opposite values
on coupled antisymmetric orbits. Thus \eqref{e:laforgia2} follows as well.
\end{proof}

\begin{theorem}[Mackey-Gelfand criterion]
\label{t:3MF}
Suppose that $\lambda_\pi^\tau \sim \lambda_\pi$. Then the following conditions are equivalent.
\begin{enumerate}[{\rm (a)}]
\item $\dim \HomA_G(\lambda_\pi^\tau, \lambda_\pi) = 0$;
\item every $(\pi^\tau \times \pi)$-orbit of $G$ on $X \times X$ is symmetric;
\item every double coset $\tau(K)sK$ is $\tau$-invariant;
\item $(G,K)$ is a Gelfand pair and $C_\tau(\sigma) = 1$ for every irreducible representation $\sigma$ contained
in $\lambda_\pi$.
\end{enumerate}
\end{theorem}
\begin{proof}
The equivalences (a) $\Leftrightarrow$ (b) and (b) $\Leftrightarrow$ (c) immediately follow from Theorem
\ref{t:2MF} and Lemma \ref{l:1MF}, respectively.
Finally, the equivalence (a) $\Leftrightarrow$ (d) follows from Lemma \ref{l:11}: indeed, the hypothesis
$\lambda_\pi^\tau \sim \lambda_\pi$ guarantees that condition (2) therein cannot hold.
\end{proof}

\begin{corollary}\label{stellina}
If $\tau(K) = K$ then  the
conditions (a) - (d)  in Theorem \ref{t:3MF} are all equivalent.
\end{corollary}
\begin{proof}
This follows immediately from Corollary \ref{c:auroretta} and Theorem \ref{t:3MF} since  if $K$ is $\tau$-invariant then $\lambda_\pi^\tau \sim \lambda_\pi$. 
\end{proof}

In particular, if $\tau = \tau_{\text{\rm inv}}$, then one has $\pi^\tau = \pi$
and the following result due to A. Garsia \cite{garsia} (see also \cite[Section 4.8]{book} and \cite[Lemma 2.3]{CST2}) is an immediate consequence.

\begin{corollary}[Symmetric Gelfand pairs: Garsia's criterion]
\label{cor:3MF}
The following conditions are equivalent:
\begin{enumerate}[{\rm (a)}]
\item every $(\pi \times \pi)$-orbit of $G$ on $X \times X$ is symmetric;
\item every double coset $KsK$ is $\tau_{\text{\rm inv}}$-invariant;
\item $(G,K)$ is a Gelfand pair and every irreducible subrepresentation of $\lambda_\pi$ is real.
\end{enumerate}
\end{corollary}

\begin{corollary}[Weakly symmetric Gelfand pairs]
\label{cor:3MF-bis}
Suppose that  
\begin{equation}\label{quadratino}
g \in K\tau(g)K
\end{equation}
 for all $g \in G$. Then $\tau(K) = K$, $\lambda_\pi^\tau\sim \lambda_\pi$ and the conditions
(a), (b), (c), and  (d) in Theorem \ref{t:3MF} are verified.
\end{corollary}
 
\begin{proof} For $k \in K$, \eqref{quadratino} becomes  $k \in K\tau(k)K$,  which implies  $\tau(k) \in K$ and that $K$ is $\tau$-invariant. Then $\lambda_\pi^\tau \sim \lambda_\pi$ by Corollary  \ref{c:auroretta} and \eqref{quadratino} yields (c) in 
Theorem \ref{t:3MF}.
\end{proof}

\section{Simply reducible groups I: Mackey's criterion}

Let $G$ be a finite group. We recall that for   $n \in \NN$ we denote by $\widetilde{G}^n = \{g,g,\ldots,g): g \in G\}$ the  diagonal subgroup  of $G^n = \underbrace{G \times G \times \ldots \times G}_{n \mbox{\ \tiny times}}$.
Let $\tau \colon G \to G$ be an involutive anti-automorphism, as before. We extend it to an involutive anti-automorphism
$\tau_n \colon G^n \to G^n$ in the obvious way, namely by setting $\tau_n(g_1, g_2, \ldots, g_n) = (\tau(g_1),
\tau(g_2), \ldots, \tau(g_n))$ for all $g_1,g_2,\ldots,g_n \in G$. Observe that $\tau_n(\widetilde{G}^n) = \widetilde{G}^n$.

\begin{lemma}
\label{l:1SRGI}
Let $(\sigma_i,V_i)$, $i=1,2$, be $G$-representations and denote by $(\iota_G,\C)$ the trivial representation of $G$.
For $T \colon V_1 \to V_2'$  define $\widetilde{T} \colon V_1 \otimes V_2 \to \C$
by setting $\widetilde{T}(v_1 \otimes v_2) = T(v_1)(v_2)$ for all $v_i \in V_i$, $i=1,2$.
Then the map
\[
\begin{array}{ccc}
\Hom_G(\sigma_1, \sigma_2') & \longrightarrow &  \Hom_G(\sigma_1 \otimes \sigma_2, \iota_G)\\
T & \mapsto & \widetilde{T}
\end{array}
\]
is a linear isomorphism. In particular,
\[
\dim \Hom_G(\sigma_1, \sigma_2') = \dim \Hom_G(\sigma_1 \otimes \sigma_2, \iota_G)
\]
so that, if  $\sigma_2$ is irreducible, the multiplicity of $\sigma_2'$ in $\sigma_1$ equals the
multiplicity of $\iota_G$ in $\sigma_1 \otimes \sigma_2$.
\end{lemma}
\begin{proof}
We leave the simple proof to the reader.
\end{proof}

\begin{definition}[Mackey-Wigner]
\label{d:MW}
{\rm One says that $G$ is \emph{$\tau$-simply reducible} provided the following two conditions are satisfied:
\begin{enumerate}[{\rm (i)}]
\item $\rho_1 \otimes \rho_2$ is multiplicity-free for all $\rho_1, \rho_2 \in \widehat{G}$; 
\item $\rho^\tau \sim \rho$ for all $\rho \in \widehat{G}$.
\end{enumerate}
When $\tau = \tau_{\inv}$ condition (ii) becomes
\begin{enumerate}
\item [{\rm (ii')}] $\rho'\sim \rho$ for all $\rho \in \widehat{G}$
\end{enumerate}
and, provided that condition (i) and (ii') are both satisfied, one simply says that $G$ is \emph{simply reducible}.}
\end{definition}

\begin{theorem}[Mackey's criterion for $\tau$-simply reducible groups]
\label{t:4SRG1}
$G$ is $\tau$-simply reducible if and only if every double coset of $\widetilde{G}^3$ in $G^3$ is
$\tau_3$-invariant. In particular, $G$ is simply reducible if and only if every double coset of $\widetilde{G}^3$ in $G^3$ is invariant under the inverse involution $(\tau_{\inv})_3$. 
\end{theorem}
\begin{proof}
We use the Mackey-Gelfand criterion (Theorem \ref{t:3MF}) with $G^3$ (resp. $\widetilde{G}^3$)
in place of $G$ (resp. $K$).  Actually, we may apply Corollary \ref{stellina} because $\tau_3(\widetilde{G}^3) = \widetilde{G}^3$.
We now show that the present theorem is a particular case of   the equivalence between (c) and (d) in the Mackey-Gelfand criterion.
First of all, observe that condition (c) of Theorem \ref{t:3MF}, in the present setting, reads that every double coset of $\widetilde{G}^3$ in $G^3$ is $\tau_3$-invariant. Similarly, denoting by $\iota_{\widetilde{G}^3}$ the trivial representation of $\widetilde{G}^3$, the first part of the equivalent condition (d) of the same theorem, reads that $\Ind_{\widetilde{G}^3}^{G^3}(\iota_{\widetilde{G}^3})$, which (cf. \cite[Example 1.6.4]{book2}) coincides with the
permutation representation $\lambda$ of $G^3$ on $L(G^3/\widetilde{G}^3)$, is multiplicity free.
Let then $\rho_1, \rho_2, \rho_3 \in \widehat{G}$. Consider the representation $\rho_1 \boxtimes \rho_2 \boxtimes \rho_3 \in \widehat{G^3}$. By virtue of Frobenius' reciprocity (cf. \cite[Theorem 1.6.11]{book2}) the multiplicity of
$\rho_1 \boxtimes \rho_2 \boxtimes \rho_3$ in $\Ind_{\widetilde{G}^3}^{G^3}(\iota_{\widetilde{G}^3})$ equals the multiplicity of $\iota_{\widetilde{G}^3}$ in $\rho_1 \otimes \rho_2 \otimes \rho_3 = \Res^{G^3}_{\widetilde{G}^3}(\rho_1 \boxtimes \rho_2 \boxtimes \rho_3)$. 
By virtue of Lemma \ref{l:1SRGI} (with $\sigma_1$ (resp. $\sigma_2$) now replaced by $\rho_1 \otimes \rho_2$ (resp. $\rho_3$)) the latter equals the multiplicity of $\rho_3'$ in $\rho_1 \otimes \rho_2$. Therefore $\Ind_{\widetilde{G}^3}^{G^3}(\iota_{\widetilde{G}^3})$ is multiplicity free if and only if $\rho_1 \otimes \rho_2$
is multiplicity free for all $\rho_1, \rho_2 \in \widehat{G}$. This shows that condition (i) in Definition
\ref{d:MW} is equivalent to the first part of  (d).

Since $C_\tau(\rho_1 \boxtimes \rho_2 \boxtimes \rho_3) = C_\tau(\rho_1) C_\tau(\rho_2) C_\tau(\rho_3)$ (cf. Proposition
\ref{p:5}), the second part of condition (d) in Theorem \ref{t:3MF} holds if and only if $C_\tau(\rho_1) C_\tau(\rho_2) C_\tau(\rho_3) =1$ whenever $\rho_3'$ is contained in $\rho_1 \otimes \rho_2$ (in particular, by Theorem \ref{thmCtC}.(1), we also have $\rho \sim \rho^\tau$ for all $\rho \in \widehat{G}$).

Now, if $\rho_1 \otimes \rho_2 = \Res^{G^2}_{\widetilde{G}^2} (\rho_1 \boxtimes \rho_2)$ is multiplicity free, then by virtue of Theorem \ref{t:10} we
have that $C_\tau(\rho_1) C_\tau(\rho_2) = C_\tau(\rho_1 \boxtimes \rho_2)$ equals $C_\tau(\rho_3')$ whenever $\rho_3'$
is contained in $\rho_1 \otimes \rho_2$. Since by Lemma \ref{l:7} $C_\tau(\rho_3) = C_\tau(\rho_3')$, we deduce that
the condition
\[
C_\tau(\rho_1 \boxtimes \rho_2 \boxtimes \rho_3) = 1  \mbox{ whenever } \rho_3' \preceq \rho_1 \otimes \rho_2
\]
is equivalent to
\[
C_\tau(\rho_3)^2 = 1 \mbox{ whenever } \rho_3 \preceq \rho_1 \otimes \rho_2.
\]
Now, $C_\tau(\rho_3)^2 = 1$ if and only if $C_\tau(\rho_3) = \pm 1$ which in turn is equivalent to the condition
$\rho_3^\tau \sim \rho_3$, by virtue of Theorem \ref{thmCtC}.(1). Since the latter is nothing but condition (ii) in Definition \ref{d:MW}, this ends the proof.
\end{proof}

\section{Simply reducible groups II: Mackey's generalizations of Wigner's criterion}
\label{s:SRGII}

Let $G$ be a finite group and $\pi \colon G \to \Sym(X)$ a (not necessarily transitive) action of $G$ on a finite set $X$. As usual, we denote by ${\mathcal O}^\pi_G(X)$ the set of all $G$-orbits of $X$. Let $\alpha \in \Sym(X)$ and
suppose that ${\mathcal O}^\pi_G(X)$ is $\alpha$-invariant, that is, $\alpha(\Omega) \in {\mathcal O}^\pi_G(X)$ for all
$\Omega \in {\mathcal O}^\pi_G(X)$. Note that this condition is always satisfied whenever $\alpha$ commutes with
$\pi$, namely, $\alpha \pi(g) = \pi(g)\alpha$ for all $g \in G$. Indeed, in this case, denoting by $\Omega_x = \{\pi(g)x: g \in G\} \in {\mathcal O}^\pi_G(X)$ the orbit of a point $x \in X$, we have $\alpha(\Omega_x) = \Omega_{\alpha x}$. We denote by ${\mathcal O}^\pi_G(X)^\alpha = \{\Omega \in {\mathcal O}^\pi_G(X): \alpha(\Omega) =
\Omega\}$ the set of orbits which are (globally) fixed by $\alpha$.

The following  generalization of the classical Cauchy-Frobenius-Burnside lemma (cf. \cite[Lemma 3.11.1]{book}) is due to Mackey.

\begin{lemma}\label{l:5SR2}
Setting  $p(g) = |\{x \in X: \pi(g)x= \alpha x\}|$ for all $g\in G$, we have 
\[
\frac{1}{|G|}\sum_{g\in G}p(g) = |{\mathcal O}^\pi_G(X)^\alpha|.
\]
\end{lemma}
\begin{proof}
Let $x\in X$  and set $q(x)=|\{g\in G:\pi(g)x= \alpha x\}|$. Note that if $\Omega_x\not\in {\mathcal O}^\pi_G(X)^\alpha$,
then $q(x) = 0$. Indeed, since $\alpha(\Omega_x)\cap \Omega_x = \emptyset$, there is no $g\in G$ such that $\pi(g) x\in \Omega_x$ equals $\alpha x\in \alpha(\Omega_x)$.  On the other hand, suppose that $\Omega_x \in 
{\mathcal O}^\pi_G(X)^\alpha$. If $g \in G$ satisfies $\pi(g)x= \alpha x$, then $gk$ satisfies the same condition for all $k \in  \Stab^\pi_G(x)$;  also if $g_1,g_2\in G$ satisfy $\pi(g_1) x = \alpha x= \pi(g_2)x$ we deduce  that $g_1^{-1}g_2\in \Stab^\pi_G(x)$.
This shows that $q(x) = |\Stab^\pi_G(x)| = \frac{|G|}{|\Omega_x|}$.
We then have 
\[
\begin{split}
\sum_{g\in G}p(g) & = |\{(x,g)\in X\times G: \pi(g) = \alpha x\}|\\
& = \sum_{x\in X}q(x)\\
& = \sum_{\Omega \in  {\mathcal O}^\pi_G(X)^\alpha}\sum_{x\in \Omega}\frac{|G|}{|\Omega|}\\
& = |G|\cdot |{\mathcal O}^\pi_G(X)^\alpha|.
\end{split}
\]
\end{proof}

The classical Cauchy-Frobenius-Burnside lemma corresponds to the case when $\alpha$ is the identity map (so that, in this case, $p(g)$, $g \in G$, equals the number of fixed points of $\pi(g)$).
\par
Let $n \geq 1$ and denote by $\pi_{n+1} \colon G^{n+1} \to \Sym(G^n)$ the action defined by
\[
\pi_{n+1}(g_1,g_2, \ldots, g_n,g_{n+1})(h_1,h_2, \ldots, h_n) = (g_1 h_1 g_{n+1}^{-1}, g_2 h_2 g_{n+1}^{-1}, \ldots, g_n h_n g_{n+1}^{-1})
\]
for all $g_1,g_2, \ldots, g_n,g_{n+1},h_1,h_2, \ldots, h_n \in G$.
Note that $\pi_{n+1}$ is transitive and that the stabilizer of $(1_G,1_G, \ldots, 1_G) \in G^n$ is given by
\[
\Stab^{\pi_{n+1}}_{G^{n+1}}(1_G,1_G, \ldots, 1_G) = \widetilde{G}^{n+1}.
\]
It follows that 
\begin{equation}
\label{e:star}
G^{n+1}/\widetilde{G}^{n+1} \cong G^n
\end{equation}
as homogeneous spaces.
We also denote by $\gamma_n \colon G \to \Sym(G^n)$ the conjugacy action of $G$ on $G^n$ given by
\[
\gamma_n(g)(g_1,g_2, \ldots, g_n) = (gg_1g^{-1},gg_2g^{-1}, \ldots, gg_ng^{-1})
\]
for all $g,g_1,g_2, \ldots, g_n \in G$.
Denoting by $\varepsilon_{n+1} \colon G \to \widetilde{G}^{n+1}$ the natural bijection given by $\varepsilon_{n+1}(g) = (g,g,\ldots,g)$, we have 
\[
\gamma_n = \pi_{n+1}\vert_{\widetilde{G}^{n+1}} \circ\varepsilon_{n+1}.
\]
In other words, $\gamma_n$ coincides with the action of $\widetilde{G}^{n+1}$ on $G^n$ and therefore (see Remark \ref{remup}) the map 
\begin{equation}
\label{e:filppo-triste}
\begin{array}{ccc}
\widetilde{G}^{n+1} \backslash G^{n+1} / \widetilde{G}^{n+1} & \longrightarrow & {\mathcal O}_G^{\gamma_n}(G^n)\\
\widetilde{G}^{n+1} (g_1,g_2, \ldots, g_n,g_{n+1}) \widetilde{G}^{n+1} & \mapsto & \Omega_{(g_1g_{n+1}^{-1},g_2g_{n+1}^{-1}, \ldots, g_ng_{n+1}^{-1})}
\end{array}
\end{equation}
is well defined and bijective. Indeed,  the orbit corresponding to the double coset of $(g_1,g_2, \ldots, g_n,g_{n+1}) \in  G^{n+1}$ contains the element $\pi_{n+1}(g_1,g_2,\ldots, g_{n+1})(1_G,1_G, \ldots, 1_G) = (g_1g_{n+1}^{-1}, g_2g_{n+1}^{-1}, \ldots, g_ng_{n+1}^{-1})$.
\par

For every $g \in G$, we now denote by $v(g)$ the cardinality of the centralizer of $g$ in $G$, that is, the number of elements $h \in G$ such that $hg = gh$. We then have

\begin{theorem}
\label{t:6SR2}
Let $n \geq 1$. The following quantities are all equal:
\begin{enumerate}[{\rm (a)}]
\item $\frac{1}{|G|} \sum_{g \in G} v(g)^n$
\item the number of double $\widetilde{G}^{n+1}$ cosets in $G^{n+1}$
\item the number of $G$-orbits on $G^n$ with respect to the action $\gamma_n$.
\end{enumerate}
\end{theorem}
\begin{proof}
We first observe that the cardinality of the centralizer of $(g,g, \ldots, g)\in \widetilde{G}^n$ in
$G^n$ is $v(g)^n$. Moreover, this is equal to the cardinality of the set of fixed points of $\gamma_n(g)$ in $G^n$.
Then, applying Lemma \ref{l:5SR2} with $X = G^n$, $G = \widetilde{G}^n$, $\pi = \gamma_n$, and $\alpha = \Id_{G^n}$,
we deduce the equality between (a) and (c) (actually, we have used the classical Cauchy-Frobenius-Burnside formula).
Finally, the equality between (b) and (c) directly follows from the bijection \eqref{e:filppo-triste}.
\end{proof}

Let $\tau$ be, as usual, an involutory anti-automorphism of $G$. For every $g \in G$, we denote by $\zeta_\tau(g)$ the
number of elements $h \in G$ such that $\tau(h^{-1})h = g$; in formul\ae
\begin{equation}\label{athos}
\zeta_\tau(g) = |\{h\in G: \tau(h^{-1})h= g\}|
\end{equation}
 Note that if $\tau = \tau_{\inv}$ then, simply writing
$\zeta(g)$ instead of $\zeta_{\tau_{\inv}}(g)$, we have that $\zeta(g) = |\{h \in G: h^2 = g\}|$.
As before, we denote by $\tau_n \colon G^n \to G^n$ its natural extension. Since $\widetilde{G}^n$  is
$\tau_n$- invariant,   the set $\widetilde{G}^n\backslash G^n /\widetilde{G}^n$ of double $\widetilde{G}^n$
cosets in $G^n$ is also $\tau_n$ invariant. Indeed, $\tau_n(\widetilde{G}^n(g_1,g_2, \ldots, g_n)\widetilde{G}^n) =
\widetilde{G}^n(\tau(g_1),\tau(g_2),\ldots, \tau(g_n))\widetilde{G}^n$, for all $g_1,g_2,\ldots, g_n \in G$.
As a consequence, $\tau_n$ induces a permutation of the double $\widetilde{G}^n$ cosets in $G^n$; we shall call
the corresponding fixed points $\tau_n$ -\emph{invariant double cosets}.
Similarly, the set ${\mathcal O}^{\gamma_n}_G(G^n)$ of all $\gamma_n$ orbits of $G \cong \widetilde{G}^n$ on $G^n$
is $\tau_n$ invariant. Therefore, $\tau_n$ induces a permutation of such orbits whose fixed points we shall call
$\tau_n$- \emph{invariant $\gamma_n$-orbits}.

\begin{theorem}
\label{t:7SR2}
Let $n \geq 1$. The following quantities are all equal:
\begin{enumerate}[{\rm (a)}]
\item $\frac{1}{|G|} \sum_{g \in G} \zeta_\tau(g)^{n+1}$
\item the number of $\tau_n$-invariant double $\widetilde{G}^{n+1}$ cosets in $G^{n+1}$
\item the number of $\tau_n$-invariant  $\gamma_n$-orbits on $G^n$.
\end{enumerate}
\end{theorem}
\begin{proof} We start by proving the equality between (b) and (c). It suffices to show that the bijective correspondence
\eqref{e:filppo-triste} transforms $\tau_n$-invariant double cosets into $\tau_n$-invariant orbits. Now, the
double coset containing the element $(g_1,g_2, \ldots, g_n)$ is $\tau_n$-invariant if and only if
\begin{equation}
\label{e:pioV1}
\exists h_1, h_2 \in G \mbox{ such that } \tau(g_i) = h_1g_ih_2 \ \ \mbox{ for all } i=1,2,\ldots,n+1
\end{equation}
while, the orbit containing $(g_1g_{n+1}^{-1}, g_2g_{n+1}^{-1}, \ldots, g_ng_{n+1}^{-1})$ is $\tau_n$-invariant if and only if
\begin{equation}
\label{e:pioV2}
\exists h_3 \in G \mbox{ such that } \tau(g_{n+1})^{-1}\tau(g_i) = h_3g_ig_{n+1}^{-1}h_3^{-1} \ \ \mbox{ for all } i=1,2,\ldots,n.
\end{equation}
Let us show that conditions \eqref{e:pioV1} and \eqref{e:pioV2} are both equivalent to
\begin{equation}
\label{e:pioV3}
\exists h_4 \in G \mbox{ such that } \tau(g_{n+1})^{-1}\tau(g_i) = h_4g_{n+1}^{-1}g_ih_4^{-1} \ \ \mbox{ for all } i=1,2,\ldots,n.
\end{equation}
Indeed, \eqref{e:pioV1} implies \eqref{e:pioV3} by taking $h_4 = h_2^{-1}$, while if \eqref{e:pioV3} holds we have
\[
\tau(g_1)h_4 g_1^{-1} = \tau(g_2)h_4g_2^{-1} = \cdots = \tau(g_{n+1})h_4 g_{n+1}^{-1}
\]
and therefore \eqref{e:pioV1} holds with $h_1= \tau(g_{n+1})h_4 g_{n+1}^{-1}$ and $h_2 = h_4^{-1}$. Finally the equivalence
between \eqref{e:pioV2} and \eqref{e:pioV3} trivially follows from the identity
\[
h_4g_{n+1}^{-1}g_ih_4^{-1} = (h_4 g_{n+1}^{-1})g_i g_{n+1}^{-1} (h_4 g_{n+1}^{-1})^{-1}
\]
which shows that the relation between $h_3$ and $h_4$ is simply given by $h_3 = h_4g_{n+1}^{-1}$.
This completes the proof of the equality of (b) and (c).

We now turn to prove that (a) equals (c). For $g \in G$ let us set
\[
p_n(g) = |\{(g_1,g_2, \ldots, g_n)\in G^n: \gamma_n(g)(g_1,g_2,\ldots, g_n) = \tau_n(g_1,g_2, \ldots, g_n)\}|
\]
so that, in particular, $p_1(g) = |\{h \in G: ghg^{-1} = \tau(h)\}|$. It is obvious that $p_n(g) = p_1(g)^n$.
Moreover, we have $ghg^{-1} = \tau(h)$ if and only if $\tau(g)^{-1}g = \tau(gh^{-1})^{-1}gh^{-1}$ and the
number of elements $h \in G$ satisfying the latter identity equals the number of elements $u \in G$ such that
$\tau(g)^{-1}g = \tau(u)^{-1}u$ (just take $u = gh^{-1}$). In other words, we have $p_1(g) = \zeta_\tau(\tau(g)^{-1}g)$.
Thus
\[
\begin{split}
\sum_{g \in G} p_n(g) & = \sum_{g \in G} p_1(g)^n\\
& = \sum_{g \in G} \zeta_\tau(\tau(g)^{-1}g)^n\\
& = \sum_{t \in G} \zeta_\tau(t)^n |\{g \in G: \tau(g)^{-1}g = t\}|\\
& = \sum_{t \in G} \zeta_\tau(t)^{n+1}.
\end{split}
\]
Then the equality of (a) and (c) follows from Lemma \ref{l:5SR2} with $X = G^n$, $G = \widetilde{G}^n$, $\pi = \gamma_n$,  $\alpha = \tau_n$  and, obviously, $p = p_n$.
\end{proof}

Recall that a conjugacy class is {\it ambivalent} when it is invariant with respect to $\tau_\inv$. Moreover the group
is {\it ambivalent} when every conjugacy class is ambivalent, equivalently, every element is conjugate to its inverse.
Then with the above notation, when $n=1$ we have:

\begin{corollary}
\label{c:8SR2}
$\frac{1}{|G|} \sum_{g \in G} \zeta_\tau(g)^2$ equals the number of $\tau$-invariant conjugacy classes of $G$. In particular (when $\tau = \tau_{\inv}$), $\frac{1}{|G|} \sum_{g \in G} \zeta(g)^2$ equals the number of ambivalent conjugacy classes of $G$. 
\end{corollary}

A more complete formulation of Corollary \ref{c:8SR2} will be given in Theorem \ref{tulliosenzasperanza}.

Also, from Theorem \ref{t:6SR2} and Theorem \ref{t:7SR2} we deduce:

\begin{corollary}
\label{c:9SR2}
We have
\[
\sum_{g \in G}\zeta_\tau(g)^{n+1} \leq \sum_{g \in G} v(g)^n.
\]
Moreover equality holds if and only if every double $\widetilde{G}^{n+1}$ coset in $G^{n+1}$ is
$\tau_n$-invariant (equivalently, if and only if every $\gamma_n$-orbit of $G$ on $G^n$ is
$\tau_n$-invariant).
\end{corollary}

When $n=2$ Theorem \ref{t:4SRG1} yields the following remarkable criterion.

\begin{corollary}[Mackey-Wigner criterion for simple reducibility]
\label{c:10SR2}
The group $G$ is $\tau$-simply reducible if and only if
\begin{equation}
\label{e:orly}
\sum_{g \in G}\zeta_\tau(g)^{3} = \sum_{g \in G} v(g)^2.
\end{equation}
In particular (Wigner's criterion), $G$ is simply reducible if and only if
$\sum_{g \in G}\zeta(g)^{3} = \sum_{g \in G} v(g)^2$. \hfill \qed
\end{corollary}

We now examine in detail the case $n=1$.

\begin{theorem}
\label{t:11SR2}
The following conditions are equivalent:
\begin{enumerate}[{\rm (a)}]
\item $\sum_{g \in G} \zeta_\tau(g)^{2} = \sum_{g \in G} v(g)$;
\item every conjugacy class of $G$ is $\tau$-invariant;
\item $\rho \sim \rho^\tau$ for every irreducible representation $\rho$ of $G$.
\end{enumerate}
\end{theorem}
\begin{proof}
The equivalence (a) $\Leftrightarrow$ (b) is a particular case of Corollary \ref{c:8SR2} since
$\frac{1}{|G|}\sum_{g \in G} v(g)$ equals the number of conjugacy classes of $G$ (cf. Theorem \ref{t:6SR2}).

Recall that the permutation representation $L(G) = L(G^2/\widetilde{G}^2)$ equals the induced representation 
$\Ind^{G^2}_{\widetilde{G}^2} \iota_{\widetilde{G}^2}$.  Moreover  this representation is multiplicity free (i.e. 
$(G^2, \widetilde{G}^2)$ is a Gelfand pair) and its decomposition into irreducibles is: 
\begin{equation}
\label{penali}
\Ind^{G^2}_{\widetilde{G}^2} \iota_{\widetilde{G}^2} \sim \bigoplus_{\rho \in \widehat{G}} (\rho' \boxtimes \rho);
\end{equation}
see \cite[Section 9.5]{book} and \cite[Corollary 2.16]{book2}.
We now observe that $(\rho')^\tau = (\rho^\tau)'$: indeed $\rho' = \rho^{\tau_{\inv}}$ and since $\tau$ and $\tau_{\inv}$ commute, \eqref{e:tau-omega} holds. Moreover since $\widehat{G} = \{\rho^\tau: \rho \in \widehat{G}\}$, from 
\eqref{penali} we deduce
\begin{equation}
\label{penali2}
\Ind^{G^2}_{\widetilde{G}^2} \iota_{\widetilde{G}^2} \sim \bigoplus_{\rho \in \widehat{G}} \left((\rho')^\tau \boxtimes \rho^\tau \right).
\end{equation}
By virtue of Proposition \ref{p:5} and Lemma \ref{l:7} we have
\[
C_\tau((\rho')^\tau \boxtimes \rho^\tau) = C_\tau((\rho')^\tau) C_\tau(\rho^\tau) = C_\tau(\rho)^2.
\]
Since by Theorem \ref{thmCtC} $\rho \sim \rho^\tau$ if and only if $C_\tau(\rho) = \pm 1$, we deduce that
this holds if and only if $C_\tau((\rho')^\tau \boxtimes \rho^\tau) = 1$.
As a consequence, the equivalence (b) $\Leftrightarrow$ (c) follows from the Mackey-Gelfand criterion
(Theorem \ref{t:3MF}) applied to \eqref{penali2}, also taking into account the equality of the quantities (b) and (c)
in Theorem \ref{t:7SR2}.
\end{proof}

The remaining part of this section is devoted to the analysis of the consequences when equality occurs in Corollary \ref{c:9SR2} for $n\geq 3$. We need two auxiliary lemmas.

\begin{lemma}\label{n0n}
If there exists a positive integer $n_0$ such that $\sum_{g\in G}\zeta_\tau(g)^{n_0+1}=\sum_{g\in G}v(g)^{n_0}$ then we also have $\sum_{g\in G}\zeta_\tau(g)^{n+1}=\sum_{g\in G}v(g)^{n}$ for all $n\leq n_0$.
\end{lemma}
\begin{proof}
By Theorem  \ref{t:6SR2} and  Theorem \ref{t:7SR2} we have $\sum_{g\in G}\zeta_\tau(g)^{n_0+1}=\sum_{g\in G}v(g)^{n_0}$ if and only if every $\gamma_{n_0}$-orbit on $G^{n_0}$ is $\tau_{n_0}$-invariant. This is equivalent to saying that for each choice of $g_1,g_2,\dotsc, g_{n_0}\in G$ there exists $g\in G$ such that $\tau(g_i)=gg_ig^{-1}$, $i=1,2,\dotsc,n_0$. But this implies the $\tau_n$-invariance of the $\gamma_n$-orbits also for all $n\leq n_0$ and
another application of the two above mentioned theorems completes the proof.
\end{proof}

\begin{lemma}\label{lemmabanale}
Let $\sigma$ be a representation. Suppose that $\sigma\otimes\sigma'$ contains the trivial representation exactly once. Then $\sigma$ is irreducible. 
\end{lemma}
\begin{proof}
Let $\sigma=\oplus_{i=1}^m\rho_i$ denote the decomposition into irreducibles of $\sigma$. By applying Frobenius
reciprocity to \eqref{penali} we deduce that each $\rho_i\otimes\rho_i'$ contains the trivial representation exactly once.
Then $\sigma\otimes\sigma'$, which contains $\rho_i\otimes\rho_i'$ for all $i=1,2,\dotsc,m$, also contains at least $m$ copies of the trivial representation. By our assumptions, this forces $m=1$, yielding the irreducibility of $\sigma$.
\end{proof}

\begin{theorem}
The following conditions are equivalent:
\begin{enumerate}[{\rm (a)}]

\item\label{ngeq3}
There exists an integer $n\geq 3$ such that $\sum_{g\in G}\zeta_\tau(g)^{n+1}=\sum_{g\in G}v(g)^n$ for all $g\in G$.

\item\label{ngeq1}
For all  $n\geq 1$ and $g\in G$ we have $\sum_{g\in G}\zeta_\tau(g)^{n+1}=\sum_{g\in G}v(g)^n$.

\item\label{Gabel}
The group $G$ is abelian and $\tau$ is the identity.

\end{enumerate}
\end{theorem}
\begin{proof}
Clearly, \eqref{ngeq1} implies \eqref{ngeq3} and, by Corollary \ref{c:9SR2}, \eqref{Gabel} implies \eqref{ngeq1}. Now assume \eqref{ngeq3}. By Lemma \eqref{n0n} the identity is verified also for $n=3$ and therefore Corollary \ref{c:9SR2} ensures that the double $\widetilde{G}^4$ cosets in $G^4$ is $\tau_3$-invariant. Then we may apply the Mackey-Gelfand criterion (Theorem \ref{t:4SRG1}) deducing that $(G^4,\widetilde{G}^4)$ is a Gelfand pair. 
It follows that if $\rho,\sigma,\theta,\xi\in\widehat{G}$ then 
$\rho\boxtimes\sigma\boxtimes\theta\boxtimes\xi$ is $G^4$-irreducible and its multiplicity in $\text{Ind}_{\widetilde{G}^4}^{G^4}\iota_{\widetilde{G}^4}$ is either $0$ or $1$. In particular, the representation
\begin{equation}\label{Res4}
\text{Res}_{\widetilde{G}^4}^{G^4}[(\rho\boxtimes\sigma)\boxtimes(\rho'\boxtimes\sigma')]
\end{equation}
(which, modulo the identification of $\widetilde{G}^4$ and $G$, is equivalent to the $G$-representation $\rho \otimes
\sigma \otimes \rho' \otimes \sigma'$) contains the trivial representation at most once. Therefore, by Lemma \ref{lemmabanale}, $\rho\otimes\sigma$ is $G$-irreducible. Note that this holds {\em for all} $\rho,\sigma\in\widehat{G}$. In particular, $\rho\otimes\rho'$ is irreducible and contains the trivial representation.
Thus  $\rho\otimes\rho' \sim \iota_G$ forcing $\rho$ to be one-dimensional. It follows
that $G$ is abelian (see \cite[Exercise 3.9.11 or Section 9.2]{book}). 

Moreover, by \eqref{penali} (with $G^2$ in place of $G$, so that $\widetilde{G^2}^2 = \{(g,h,g,h):g,h \in G\}$), the representation
\[
\text{Res}^{G^4}_{\widetilde{G^2}^2}[(\rho\boxtimes\sigma)\boxtimes(\rho'\boxtimes\sigma')]
\]
contains the trivial $\widetilde{G^2}^2$-representation exactly once. 
By further restricting to the subgroup $\widetilde{G}^4$ we deduce that \eqref{Res4} contains the trivial
$\widetilde{G}^4$-representation exactly once.

By Frobenius reciprocity, this implies that $(\rho\boxtimes\sigma)\boxtimes(\rho'\boxtimes\sigma')$ is contained in $\text{Ind}_{\widetilde{G}^4}^{G^4}\iota_{\widetilde{G}^4}$ (with multiplicity one). Then Theorem \ref{t:3MF}.(d) ensures that
\[
1 = C_\tau\left((\rho\boxtimes\sigma)\boxtimes(\rho'\boxtimes\sigma')\right) = C_\tau(\rho)^2C_\tau(\sigma)^2,
\]
where the second equality follows from Proposition \ref{p:5} and Corollary \ref{ppii}.
This implies $C_\tau(\rho)=\pm1$ so that Theorem
\ref{thmCtC}.(1) ensures the equivalence $\rho\sim\rho^\tau$, for all $\rho\in \widehat{G}$. Since all these representations are one-dimensional, the latter means exactly that $\rho(g)=\rho(\tau(g))$ for all $g\in G$ and $\rho\in\widehat{G}$. Since $G$ is abelian (so that $\widehat{G}$ separates the elements in $G$), this in turn gives $\tau = \Id_G$.
\end{proof}

Note that if in the preceding theorem $\tau$ is the inversion, then $G$ must be a direct product of cyclic groups
of order two.

\section{An example: the Clifford groups}
\label{s:Clifford}
In this section, as an application of Mackey's criterion (Theorem  \ref{t:4SRG1}) and, independently,
of the Mackey-Wigner criterion (Corollary \ref{c:10SR2}), we show that the Clifford groups $\CL(n)$ are $\tau$-simply reducible (where the involutive anti-automorphism $\tau$ of $\CL(n)$ is suitably defined according to the congruence  class of $n$ modulo $4$).

The \emph{Clifford group} $\CL(n)$, $n \geq 1$, is the group generated by the elements 
$\varepsilon, \gamma_1, \gamma_2, \ldots, \gamma_n$ with defining relations (called \emph{Clifford relations})
\begin{equation}
\label{e:cliff-rel}
\begin{split}
\varepsilon^2 & = 1\\
\gamma_i^2 & = 1\\
\gamma_i\gamma_j & = \varepsilon \gamma_j\gamma_i\\
\end{split}
\end{equation}
for all $i,j=1,2,\ldots,n$ such that $i \neq j$. For $n=1$ one should also add the relation
\[
\varepsilon \gamma_1 = \gamma_1 \varepsilon.
\]
(Note that for $n \geq 2$ the relations $\varepsilon \gamma_i = \gamma_i \varepsilon$, $i=1,2,\ldots,n$,
are easily deduced from \eqref{e:cliff-rel}).  

The \emph{Clifford-Littlewood-Eckmann group} $G_{s,t}$, $s, t \in \NN$,  is the group with
generators $\varepsilon, a_1,a_2,\ldots,a_s, b_1, b_2,\ldots,b_t$ and the following defining relations: 
$\varepsilon^2=1$; $a^2_i=\varepsilon$, $b^2_j=1$, $b_j\varepsilon=\varepsilon b_j$, and $a_ib_j=\varepsilon b_ja_i$
for all $i,j$; and $a_ia_j=\varepsilon a_ja_i$ and $b_ib_j=\varepsilon b_jb_i$ for all $i\neq j$ (see \cite{lam}). 
It can be easily shown that $G_{s,t}$ is a finite group of order $2^{s+t}$.    
Note also that $G_{0,n} = \CL(n)$ for all $n \in \NN$.

The groups $G_{s,t}$ are implicit in W.K. Clifford's work on ``geometric algebra'' \cite{cliffordWK}. Indeed,
$G_{s,t}$ appears naturally as a subgroup of the group of units of the Clifford algebra $C(\varphi_{s,t})$ of the quadratic form $\varphi_{s,t} := s \langle -1\rangle \perp t\langle 1\rangle$ over any field of characteristic 
$\neq 2$. They where explicitly defined by D.E. Littlewood \cite{littlewood} in 1934.
These groups are of great interest to theoretical physicists. 
For example, $G_{0,3} = \CL(3)$ is the group generated by the three (Hermitian) \emph{Pauli spin matrices}
(coming from the commutation relations between angular momentum operators in the study of the spin of the electron)
and $G_{0,4}= \CL(4)$ is  the \emph{Dirac group}, generated by the four (Hermitian) \emph{Dirac matrices}
(defined by Dirac  \cite{dirac} in his study of the relativistic wave equation). 
More generally, the groups $G_{0,2n} = \CL(2n)$, $n \geq 1$, arise naturally in quantum field theory (e.g. in the theory of Fermion fields): originally they were introduced by Jordan and Wigner in their paper on {\it Pauli's Exclusion Principle} \cite{JW}. Using Frobenius-Burnside theory of finite group representations (cf.
\cite[Section 3.11]{book}) they determined all irreducible representations of these groups: apart the
$2^n$ one-dimensional representations, $G_{0,2n}$ has only one irreducible representation  
(of dimension $2^n$).                       

For the sake of completeness, we mention that the groups $G_{s,0}$, $s \geq 1$, are also important to
physicists. For instance, they were studied by Jordan, von Neumann and Wigner \cite{JvNW} in connection with
their algebraic formalism for the mathematical foundations of quantum mechanics. Eddington \cite{E1,E2},
in his studies in astrophysics, considered sets of anticommuting matrices and complex representations of the
group $G_{5,0}$. Note also that the groups $G_{s,0}$ play an important role in connection with the 
{\it Hurwitz problem} of composition of quadratic forms \cite{h1,h2,radon}. 
Indeed, Eckmann \cite{E} rediscovered these groups and observed that a set of solutions to the Hurwitz equations over a field $\FF$ corresponds to an $n$-dimensional orthogonal representation $\rho$ of $G_{s,0}$ satisfying $\rho(\varepsilon) = - I_n$. Then, Eckmann determined all irreducible orthogonal representations of $G_{s,0}$ over the real field $\RR$ and deduced a purely group theoretical proof of the Hurwitz-Radon theorem on the composition of sums of squares.

Returning back to our investigations, we shall make use of the following alternative description of the Clifford 
groups (cf. \cite[Chapter 4]{Simon}). Setting $X = \{1,2,\ldots,n\}$, we have  
$\CL(n) = \{\pm\gamma_A: A \subseteq X\}$ with multiplication given by
\begin{equation}
\label{e:mult-clifford}
\varepsilon_1 \gamma_A \cdot \varepsilon_2 \gamma_B = \varepsilon_1 \varepsilon_2 (-1)^{\xi(A,B)} \gamma_{A \triangle B}
\end{equation}
where $\triangle$ denotes the symmetric difference of two sets and  $\xi(A,B)$ equals the number of elements $(a,b) \in A \times B$ such that $a > b$, for all $\varepsilon_1, \varepsilon_2 \in \{1,-1\}$ and $A,B \subseteq X$. Notice that the identity element is given by $1_{\CL(n)} = \gamma_\varnothing$ and
that $(\varepsilon\gamma_A)^{-1} = \varepsilon (-1)^\frac{|A|(|A|-1)}{2}\gamma_{A}$ for all $\varepsilon = \pm 1$
and $A \subseteq X$.

Consider now the map $\tau' \colon \CL(n) \to \CL(n)$ defined by
\begin{equation}
\label{e:tau1}
\tau'(\varepsilon \gamma_A) = \varepsilon (-1)^\frac{|A|(|A|+1)}{2}\gamma_{A}
\end{equation}
and note that $\tau' = \tau_{\inv} \circ \tau''$ where 
\[
\tau''(\varepsilon \gamma_A) = \varepsilon (-1)^{|A|}\gamma_{A}
\] 
for all $\varepsilon = \pm 1$ and $A \subseteq X$.
It is straightforward to check that $\tau''$ is an involutive automorphism of $\CL(n)$ so that
$\tau'$ is an involutive anti-automorphism of $\CL(n)$.
We then set 
\[
\tau = 
\begin{cases}
\tau' & \mbox{ if } n \equiv 3 \mod 4\\
\tau_{\inv} & \mbox{ otherwise}.
\end{cases}
\]

\begin{theorem} 
\label{t:cln-wsr}
The group $\CL(n)$ is $\tau$-simply reducible.
\end{theorem}
\begin{proof} We present two different proofs: the first one making use of the Mackey criterion (Theorem \ref{t:4SRG1}) and the second one based on the Mackey-Wigner criterion (Corollary \ref{c:10SR2}).\\

\noindent
{\it First proof.} Set $G=\CL(n)$ and let us check that every double coset of $\widetilde{G}^3$ in $G^3$ is $\tau_3$-invariant.
Consider the action $\pi$ of $G^3$ on $G^2$ given by
\[
\pi(g_1,g_2,g_3)(h_1,h_2) = (g_1h_1g_3^{-1}, g_2h_2g_3^{-1})
\]
for all $g_1,g_2,g_3,h_1,h_2 \in G$. Then the stabilizer of the element $(1_G,1_G) \in G^2$ is exactly
$\widetilde{G}^3$ and the double cosets of $\widetilde{G}^3$ in $G^3$ coincide with the $\widetilde{G}^3$-orbits
of $G_2$ under the action $\pi$. 
Let $A,C \subseteq X$. Then
\begin{equation}
\label{e:18nov}
\begin{split}
\gamma_C^{-1} \gamma_A \gamma_C & = (-1)^{\frac{|C|(|C|-1)}{2} + \xi(A,C) + \xi(C, A \triangle C)} \gamma_{C \triangle(A \triangle C)}\\
& = (-1)^{\frac{|C|(|C|-1)}{2} + \xi(A,C) + \xi(C,A) + \xi(C,C)}\gamma_A\\
& =_{*} (-1)^{|A||C| - |A \cap C|}\gamma_A
\end{split}
\end{equation}
where $=_*$ follows from the fact that $\xi(C,C) = \frac{|C|(|C|-1)}{2}$. From
\eqref{e:18nov}, a case-by-case analysis yields the $\tau_3$-invariance of all $\widetilde{G}^3$-orbits. 
As an example, consider the $\widetilde{G}^3$-orbit of the element $(\gamma_X,\gamma_X)$. 
Suppose first that $n$ is odd. Then
\[
\pi(\gamma_C,\gamma_C,\gamma_C)(\gamma_X,\gamma_X) = ((-1)^{(n-1)|C|}\gamma_X,(-1)^{(n-1)|C|}\gamma_X) =
(\gamma_X,\gamma_X)
\]
for all $C \subseteq X$. Therefore the $\widetilde{G}^3$-orbit of $(\gamma_X,\gamma_X)$ reduces
to one point. Now, if $n \equiv 1 \mod 4$ we have $\tau(\gamma_X) = \gamma_X^{-1} = \gamma_X$, while if $n \equiv 3 \mod 4$ we have $\tau(\gamma_X) = -\gamma_X^{-1} = \gamma_X$. 
Similarly, when $n$ is even we have that
the $\widetilde{G}^3$-orbit of $(\gamma_X,\gamma_X)$ is $\{(\gamma_X,\gamma_X),(-\gamma_X,-\gamma_X)\}$ and
$\tau(\gamma_X) = \gamma_X^{-1} = (-1)^\frac{n(n-1)}{2} \gamma_X$. 

Thus, since all double cosets of $\widetilde{G}^3$ in $G^3$ are $\tau_3$-invariant, from the Mackey criterion (Theorem \ref{t:4SRG1}) we deduce that $\CL(n)$ is $\tau$-simply reducible.\\

\noindent
{\it Second proof.} We start by computing the Right Hand Side in \eqref{e:orly} where $G = \CL(n)$. 
First note that for $A \subseteq X$ we have, 
\begin{equation}
\label{e:i-C-even}
v(\gamma_A) = 2^n.
\end{equation}
Indeed, by virtue of \eqref{e:18nov}, we have
\[
v(\gamma_A) = \vert\{\varepsilon \gamma_C: \varepsilon = \pm 1, |C \cap A| \mbox{ is even}\}\vert =
2\cdot 2^{|A|-1} \cdot 2^{n-|A|} = 2^n
\]
if $|A|$ is even, and
\[
\begin{split}
v(\gamma_A) & = \vert\{\varepsilon \gamma_C: \varepsilon = \pm 1, |C \cap A| \mbox{ is even and } |C| \mbox{ is even}\}\vert \\
& \ \ \ \ \ \ \ \ + \vert\{\varepsilon \gamma_C: \varepsilon = \pm 1, |C \cap A| \mbox{ is odd and } |C| \mbox{ is odd}\}\vert\\
& = 2\cdot 2^{|A|-1} \cdot 2^{n-|A|-1} + 2 \cdot 2^{|A|-1} \cdot 2^{n-|A|-1}\\
& = 2^n
\end{split}
\]
if $|A|$ is odd.
We then have
\begin{equation}
\label{e:v2}
\begin{split}
\sum_{g \in G} v(g)^2 & = \sum_{A \subseteq X} v(\gamma_A)^2 +  \sum_{A \subseteq X} v(-\gamma_A)^2\\
\mbox{(since $v(-\gamma_A) = v(\gamma_A)$)} \ \ & = 2 \sum_{A \subseteq X} v(\gamma_A)^2\\
\mbox{(by \eqref{e:i-C-even})} \ \ & = 2 \sum_{A \subseteq X} 2^{2n}\\
& = 2^{3n+1}.
\end{split}
\end{equation}

We now compute the Left Hand Side in \eqref{e:orly}. First observe that by \eqref{e:mult-clifford}
we have that $\zeta_\tau(g) = 0$ if (and only if) $g \neq \pm 1_G$.
Suppose first that $n \equiv 3 \mod 4$. Then we have
\[
\begin{split}
\zeta_\tau(1_G) & = \vert \{\varepsilon \gamma_C: \varepsilon = \pm 1 \mbox{ and } \tau(\varepsilon \gamma_C)^{-1} \varepsilon \gamma_C = 1_G\}\vert\\
& = 2 \vert \{\gamma_C: \tau(\gamma_C)^{-1}\gamma_C = 1_G\}\vert\\
& = 2 \vert \{\gamma_C: (-1)^{|C|} \gamma_C \gamma_C = 1_G\}\vert\\
& = 2 \vert \{\gamma_C: (-1)^{\frac{|C|(|C|+1)}{2}}1_G = 1_G\}\vert\\
& = 2 \vert \{\gamma_C: |C|(|C|+1) \equiv 0,3 \mod 4\}\vert\\
& = 2 \cdot 2^{n-1} = 2^n.
\end{split}
\]
Analogously, one has
\[
\zeta_\tau(-1_G) = 2 \vert \{\gamma_C: |C|(|C|+1) \equiv 1,2 \mod 4\}\vert = 2^n
\]
so that, alltogether,
\begin{equation}
\label{e:zeta3}
\sum_{g \in G} \zeta_\tau(g)^3 = \zeta_\tau(1_G)^3 + \zeta_\tau(-1_G)^3 = 2^{3n+1}
\end{equation}
On the other hand, if $n \equiv 0,1,2 \mod 4$ we have
\[
\zeta_\tau(1_G) = 2 \vert \{\gamma_C: |C|(|C|-1) \equiv 0,1 \mod 4\}\vert = 2^n
\]  
and
\[
\zeta_\tau(-1_G) = 2 \vert \{\gamma_C: |C|(|C|-1) \equiv 2,3 \mod 4\}\vert = 2^n
\]
thus showing that \eqref{e:zeta3} holds also in this case.
Comparing  \eqref{e:v2} and \eqref{e:zeta3}, from the Mackey-Wigner criterion (Corollary \ref{c:10SR2}) we deduce that
$\CL(n)$ is $\tau$-simply reducible.
\end{proof}

\begin{remark}{\rm 
Let $\rho, \sigma \in \widehat{\CL(n)}$.
In \cite{AM2} we give an explicit decomposition of the tensor product $\rho \otimes \sigma$.
According to Theorem \ref{t:cln-wsr}, this is multiplicity free and, moreover, $\rho \sim \rho^\tau$.}
\end{remark}

\section{The twisted Frobenius-Schur theorem}

Let $N$ be a finite group, $\tau \colon N \to N$ an involutory anti-automorphism, and denote by $\alpha \in
\Aut(N)$ the involutory automorphism defined by $\alpha(n) = \tau(n^{-1})$ for all $n \in N$. 
Consider the semi-direct product 
\begin{equation}\label{aurorazoccola}
G = N \rtimes_\alpha \langle \alpha \rangle.
\end{equation}
In other words,  $G=\{(n, \alpha^\varepsilon): 
n \in N, \varepsilon \in \{0,1\}\}$ and 
\[
(n,\alpha^\varepsilon)(n',\alpha^{\varepsilon'}) =  (n\alpha^\varepsilon(n'), \alpha^{\varepsilon+ \varepsilon'})
\]
for all $n,n' \in N$ and $\varepsilon, \varepsilon' \in \{0,1\}$.
If we identify $N$ with the normal subgroup $\{(n, \alpha^0): n \in N\}$ and we set $h = (1_N,\alpha)$, then $G$ is generated  by $N$ and $h$ and the following relations hold: $h^2 = 1$ and  $hnh = \tau(n)^{-1}$, for all $n \in N$.
We then have the coset decomposition $G = N \coprod hN$. Moreover, we can define the {\it alternating representation} of $G$ (with respect to $N$)  as the one--dimensional   representation $(\varepsilon, \CC)$ defined by 
 \[
 \varepsilon(g) = \begin{cases}
 1 & \mbox{ if $g\in N$}\\
 -1 & \mbox{ otherwise.}
\end{cases}
 \]
 
We define  two actions of $C_2= \{1,-1\}$ on $\widehat{N}$ and $\widehat{G}$ as follows: $1$ acts trivially in both cases;
$-1$ acts on $\widehat{N}$ by $\widehat{N} \ni \sigma \mapsto \ ^h\!\sigma \in \widehat{N}$ where 
\[
^h\!\sigma(n) = \sigma(h^{-1}nh)
\]
 for all $n \in N$; finally, $-1$ acts on $\widehat{G}$ by $\widehat{G} \ni \theta \mapsto  \theta\otimes \varepsilon \in \widehat{G}$.
Clearly, both $\widehat{N}$ and  $\widehat{G}$ are partitioned into their $C_2$-orbits. Moreover every such an orbit
consists of one or two representations.
Let also   
\[
I_G(\sigma) = \{g\in G:\  ^g\!\sigma \sim \sigma\}
\]
be  the in\ae rtia group of $\sigma\in \widehat{N}$ with respect to $G$ (again, $ ^g\!\sigma(n) = \sigma(g^{-1}ng)$) and 
$$
\widehat{G}(\sigma) = \{\theta \in \widehat{G}: \sigma\preceq \Res_N^G\theta\} 
\equiv \{\theta\in \widehat{G}: \theta \preceq \Ind_N^G\sigma \}.
$$ 
The following theorem yields a very natural bijection between 
the orbits of $C_2$ on $\widehat{N}$ and those on $\widehat{G}$.
For the proof we refer to \cite[Theorem 3.1]{Clifford} and \cite[Section III.11]{Simon}.
 
\begin{theorem}\label{t;clifI2}
\begin{enumerate}[{\rm (1)}]
\item If $I_G(\sigma) = N$, then $\theta: = \Ind_N^G\sigma \in \widehat{G}$,  $\theta\otimes \varepsilon = \theta$
and $\Res^G_N\theta = \sigma \oplus ^h\!\sigma$, with $\sigma$ and $ ^h\!\sigma$  not equivalent.
\item If  $I_G(\sigma) = G$, then, taking $\theta \in \widehat{G}(\sigma)$ we have  $\Ind_N^G(\sigma) = \theta \oplus (\theta\otimes \varepsilon)$ with $\theta \not\sim \theta\otimes \varepsilon$ and 
$\Res_N^G\theta = \Res_N^G(\theta\otimes \varepsilon) = \sigma.$
\item The map
\[
\{\sigma,  ^h\!\sigma\} \mapsto  \widehat{G}(\sigma) \ \ \ \ \ \ \ \mbox{ when $I_G(\sigma) = N$}
\]
and 
\[
\{\sigma\}    \mapsto  \widehat{G}(\sigma) \ \ \ \ \ \ \ \ \ \ \ \ \mbox{ when $I_G(\sigma) = G$}
\]
yields a one--to--one correspondence  between  the $C_2$-orbits  on $\widehat{N}$ and on $\widehat{G}$. 
In particular, to each single--element orbit on $\widehat{N}$ (resp. on $\widehat{G}$) there corresponds 
a two--elements orbit on   $\widehat{G}$ (resp. on $\widehat{N}$).
\end{enumerate}
\end{theorem}

In the following, given an irreducible representation $\sigma$ of $N$, we denote by $\sigma^\tau$  and  $C_\tau(\sigma)$ the associated $\tau$-conjugate representation and the $\tau$-Frobenius-Schur indicator of $\sigma$ (as in \eqref{rhotau} and Definition \ref{rhotau2}; recall that $C = C_{\tau_{\inv}}$). 
As remarked in the Introduction, the following result goes back to Kawanaka and Mastuyama \cite{kawanaka} but
the proof follows the lines in \cite[Exercise 4.5.1]{Bump}.

\begin{theorem}[Twisted Frobenius-Schur theorem]
\label{t1TFS}
Let $\sigma$  be an irreducible representation of $N$ and denote by $\chi_\sigma$ its character. Then 
\begin{equation}\label{H;T2}
\frac{1}{|N|}\sum_{n\in N}\chi_\sigma(\tau(n)^{-1}n) = C_\tau(\sigma).
\end{equation}
\end{theorem}

\begin{proof}
Let $G$  be as in \eqref{aurorazoccola}.
We distinguish two cases.

$\underline{^h\!\sigma \sim \sigma}$.  In this case, for $\theta \in \widehat{G}(\sigma)$, by Theorem \ref{t;clifI2}.(2) we have   $\theta \not\sim \theta\otimes \varepsilon$ and
 $\Res^G_N\theta = \sigma$, so that  $\chi_\theta(n) = \chi_\sigma(n)$ for all $n \in N$. since $g^2 \in N$ for all $g \in G$, we have 
\[
\begin{split}
\sum_{g\in G}\chi_\theta(g^2) & = \sum_{n \in N}\chi_\sigma(n^2)+  \sum_{n \in N}\chi_\sigma((hn)^2)\\ 
& =  \sum_{n \in N}\chi_\sigma(n^2)+  \sum_{n \in N}\chi_\sigma(\tau(n)^{-1}n).
\end{split}
\]
By the classical Frobenius-Schur theorem \cite[Theorem 9.7.7]{book}, we have
\[
C(\theta) = \frac{1}{|G|}\sum_{g\in G}\chi_\theta(g^2)
\]
and
\[
C(\sigma) = \frac{1}{|N|}\sum_{n\in N}\chi_\sigma(n^2).
\]
Therefore, since $|G| = 2|N|$, we deduce that
\begin{equation}\label{G;T4}
C(\theta) = \frac{1}{2}C(\sigma) + \frac{1}{2|N|}\sum_{n \in N}\chi_\sigma(\tau(n)^{-1}n).
\end{equation}
Suppose that $\sigma$ is self-conjugate.

Denote by $A(n)$ (resp. $A^\tau(n)$), with $n\in N$,  a matrix realization of $\sigma$ (resp. $\sigma^\tau$). Note that $^h\!A(n):=A(h^{-1}nh)$, with $n\in N$,  is a matrix realization of $^h\!\sigma$ (cf. \cite[Lemma 3.1]{Clifford}).
Moreover,  
\begin{equation}\label{2;T4}
A^\tau(n) = \overline{ ^h\!A(n)}
\end{equation}
for all $n \in N$. 
Indeed, for all $n\in N$ we have $ ^h\!\sigma(n) = \sigma(\tau(n)^{-1})$  and therefore 
\begin{equation}\label{fufu}
\sigma^\tau(n) = \sigma[\tau(n)]^T =  ^h\!\sigma(n^{-1})^T = ( ^h\!\sigma)'(n)
\end{equation}
so that $A^\tau(n) = {^h\!A(n^{-1})^T} = \overline{^h\!A(n)}$.
Let also $M(g)$, $g \in G$, denote  a matrix realization of $\theta$ such that $M(n) = A(n)$ for all $n \in N$.
Then, the unitary matrix $V=M(h)$ satisfies $M(nh) = A(n)V$ for all $n \in N$, $V^2 = M(h^2) = M(1_G) = I$ so that 
\begin{equation}
\label{auroraseimia0}
V^* = V \  \mbox{ and } \ V^T = \overline{V}
\end{equation} 
and 
\begin{equation}
\label{auroraseimia}
^h\!A(n) = V^*A(n)V = V A(n) V
\end{equation}
for all $n \in N$. 
Since $\sigma$ is self-conjugate, we can find a unitary matrix $W$ such that
\begin{equation}
\label{auroraseimia2}
\overline{A(n)} = WA(n)W^*
\end{equation}
for all $n \in N$. We therefore obtain
\[
\begin{split}
A^\tau(n) & = \overline{^h\!A(n)} \hspace{2.5cm} \mbox{(by \eqref{2;T4})}\\
& = \overline{VA(n)V} \hspace{2cm} \mbox{(by \eqref{auroraseimia})}\\
& = V^T \ \overline{A(n)} V^T  \hspace{1.5cm}  \mbox{(by \eqref{auroraseimia0})}\\
& = V^T WA(n)W^* V^T \ \ \ \mbox{(by \eqref{auroraseimia2})}
\end{split}
\]
that is, 
\begin{equation}
\label{auroraseimia3}
W^*V^TA^\tau(n) = A(n)W^*V^T
\end{equation}
for all $n \in N$. Since $W^*\overline{V}W = \pm V$ (cf. \cite[Theorem 3.4)]{Clifford}), and  applying also \eqref{auroraseimia0},
we get
\begin{equation}
\label{auroraseimia4}
(W^*V^T)^T = V \overline{W} = \pm \overline{W}\overline{V} = \pm \overline{W}V^T.
\end{equation}
>From Lemma \ref{lemmaAAT}, it follows  that $W\overline{W} = \pm I$. More precisely,
$W\overline{W} = I$ (resp. $W\overline{W} = -I$) if $\sigma$ is real (resp. quaternionic). Combining
with Theorem \ref{thmCtC}, this is equivalent to
\begin{equation}
\label{auroraseimia5}
\overline{W} = C(\sigma) W^*.
\end{equation}
>From \eqref{auroraseimia4} and \eqref{auroraseimia5} we obtain
\begin{equation}
\label{auroraseimia6}
(W^*V^T)^T = \pm C(\sigma) W^* V^T
\end{equation}
where the sign is the same as in $W^*\overline{V}W = \pm V$.
Thus, from \eqref{auroraseimia3} ($W^*V^T$ intertwines $A^\tau$ and $A$)  and \eqref{auroraseimia6} ($W^*V^T$ is symmetric/antisymmetric)  with the same sign therein, we obtain
\begin{equation}
\label{auroraseimia7}
C_\tau(\sigma) = \pm C(\sigma).
\end{equation}
Now, if in \eqref{auroraseimia7} the sign is $+$, then by \cite[Theorem 3.4.(2)]{Clifford} we have
$C(\theta) = C(\sigma)$ and therefore from \eqref{G;T4} and \eqref{auroraseimia7} we deduce that
\eqref{H;T2} is satisfied. On the other hand, if in \eqref{auroraseimia7} the sign is $-$, then
$\theta$ is complex, $C(\theta) = 0$ and $C_\tau(\sigma) = -C(\sigma)$ and from
 \eqref{G;T4} we again deduce \eqref{H;T2}. This completes the proof in the case $\sigma$ is selfconjugate.
 
Suppose now that $\sigma$ is complex. Then by \cite[Theorem 3.4.(1)]{Clifford} we have that
$\theta$ is complex as well. Moreover $\sigma' \not\sim \sigma$ and $^h\!\sigma \sim \sigma$
imply that $\sigma^\tau \equiv\  ^h\!\sigma' \sim \sigma' \not\sim \sigma$ (recall \eqref{fufu}) and therefore $C_\tau(\sigma) = 0$. Again, from \eqref{G;T4} we deduce \eqref{H;T2}.
This completes the proof in the case $\sigma$ is complex and, together with the previous step
completes the proof for the case $^h\!\sigma \sim \sigma$.

We now discuss the remaining case.

$\underline{^h\!\sigma \not\sim \sigma}$. From Theorem \ref{t;clifI2}.(1) we deduce that $\Res^G_N\theta = \sigma\oplus ^h\!\sigma$ and therefore
\[
\begin{split}
\sum_{g \in G} \chi_\theta(g^2) & = \sum_{n \in N} \chi_\theta(n^2) + \sum_{n \in N} \chi_\theta(\tau(n)^{-1}n)\\
& = \sum_{n \in N} \chi_\sigma(n^2) + \sum_{n \in N} \chi_{^h\!\sigma}(n^2) +
\sum_{n \in N} \chi_\sigma(\tau(n)^{-1}n) + \sum_{n \in N} \chi_{^h\!\sigma}(\tau(n)^{-1}n).
\end{split}
\]
As 
\[
\chi_{^h\!\sigma}(n^2) = \chi_{\sigma}(hn^2h) = \chi_{\sigma}(\tau(n)^{-2})
\]
and
\[
\chi_{^h\!\sigma}(\tau(n)^{-1}n) = \chi_{\sigma}(h\tau(n)^{-1}nh) 
= \chi_\sigma\left(\tau((\tau(n)^{-1}n)^{-1})\right)= \chi_\sigma(n\tau(n)^{-1})
\]
from the fact that $n \mapsto \tau(n)^{-1}$ is an automorphism, we deduce that
\begin{equation}
\label{auroraseimia8}
C(\theta) = C(\sigma) + \frac{1}{|N|} \sum_{n \in N} \chi_\sigma(\tau(n)^{-1}n).
\end{equation}

Suppose that $\sigma$ is real (resp. quaterionionic). Then,
by virtue of \cite[Theorem 3.3.(1)]{Clifford} $\theta$ is real (resp. quaterionionic) as well and
therefore $C(\sigma) = C(\theta)$. Now, if $\sigma \sim \sigma'$, since by hypothesis $^h\!\sigma
\not\sim \sigma$, and \eqref{fufu} holds also in this case, we have $\sigma^\tau \equiv\  ^h\!\sigma' \not\sim \sigma' \sim \sigma$, and therefore $C_\tau(\sigma) = 0$. Then \eqref{H;T2} follows from \eqref{auroraseimia8}.

Suppose now that $\sigma$ is complex. If $\sigma' \not\sim\  ^h\!\sigma$, from \cite[Theorem 3.3.(2)]{Clifford} we deduce that $\theta$ is complex as well. Moreover, $\sigma^\tau 
\equiv\  ^h\!\sigma' \not\sim  \sigma$ and therefore $C(\sigma) = C(\theta) = C_\tau(\sigma) = 0$
and \eqref{H;T2} follows again from \eqref{auroraseimia8}.

Finally, if $\sigma$ is complex and $\sigma' \sim\  ^h\!\sigma$, then 
\cite[Theorem 3.3.(3)]{Clifford} ensures that $\theta$ is selfconjugate. Moreover,
$\sigma^\tau  \equiv\  ^h\!\sigma' \sim \sigma$. We then denote by $U$ an intertwining unitary matrix such that
\begin{equation}
\label{auroraseimia10}
U A^\tau(n) = A(n) U
\end{equation}
for all $n \in N$ ($A(n)$ is as in \eqref{2;T4}). Since $h^2 = 1_G$, from \cite[Theorem 3.3.(3)]{Clifford} we deduce that
$\overline{U}U = \pm I$, that is,
\begin{equation}
\label{auroraseimia11}
U = \pm U^T.
\end{equation}
Now, if in \eqref{auroraseimia11} the sign is $+$, \cite[Theorem 3.3.(3)]{Clifford} ensures that
$\theta$ is real, while \eqref{auroraseimia10} and \eqref{auroraseimia11} give $C_\tau(\sigma) = 1$.
In other words, $C(\sigma) = 0$, $C_\tau(\sigma) = C(\theta) = 1$ and, once more,
\eqref{H;T2} follows from \eqref{auroraseimia8}. Similarly if the sign in \eqref{auroraseimia11} is $-$.
\end{proof}

Recall that $\zeta_\tau(n)$, $n \in N$, denotes the number of elements $m \in N$ such that $\tau(m^{-1})m = n$ (cf. \eqref{athos}).

\begin{corollary}
For all $n \in N$ we have 
\label{c:auroraseimia}
\begin{equation}
\label{auroraseimia12}
\zeta_\tau(n) = \sum_{\sigma \in \widehat{N}} C_\tau(\sigma) \chi_\sigma(n).
\end{equation}
In particular, 
\[
\zeta_\tau(1_N) = \sum_{\substack{\sigma \in \widehat{N}:\\ C_\tau(\sigma) = 1}} d_\sigma -
\sum_{\substack{\sigma \in \widehat{N}:\\ C_\tau(\sigma) = -1}} d_\sigma.
\]
\end{corollary}
\begin{proof}
We observe that $\zeta_\tau$ is a central function. Indeed, if $m,n,s \in N$ and 
$\tau(m^{-1})m = n$, then
\[
sns^{-1} = s \tau(m^{-1}) \tau(s)\tau(s^{-1}) m s^{-1} = \tau[\tau(s^{-1}) m s^{-1}]^{-1}\tau(s^{-1}) m s^{-1}.
\]
Therefore the map $m \mapsto \tau(s^{-1}) m s^{-1}$ yields a bijection between the set of
solutions of $\tau(m^{-1})m = n$ and the set of solutions of $\tau(m_s^{-1})m_s = sns^{-1}$.
>From the Frobenius-Schur twisted formula \eqref{H;T2} we deduce
\begin{equation}
\label{auroraseimia13}
\frac{1}{|N|} \sum_{n \in N} \overline{\chi_\sigma(n)}\zeta_\tau(n) = C_\tau(\sigma)
\end{equation}
(observe that $\zeta_\tau(n)$ and $C_\tau(\sigma)$ are both real). Since $\zeta_\tau$ is central,
by the orthogonality relation for characters (cf. \cite[Equation (3.21)]{book}) \eqref{auroraseimia12}
immediately follows from \eqref{auroraseimia13}.
\end{proof}
We are now in position to complete Corollary \ref{c:8SR2} by adding a  third representation theoretic quantity.

\begin{theorem}\label{tulliosenzasperanza}
The following quantities are equal:
\begin{enumerate}[{\rm (a)}]
\item{the number of $\sigma \in \widehat{N}$ such that $\sigma \sim \sigma^\tau$;}
\item{the number of $\tau$-invariant conjugacy classes of $N$;}
\item{$\frac{1}{|N|} \sum_{n \in N} \zeta_\tau(n)^2$.}
\end{enumerate}
\end{theorem}
\begin{proof}
The equality between the numbers in (b) and (c) corresponds to Corollary \ref{c:8SR2}. On the other hand,
>From Corollary \ref{c:auroraseimia} we also have
\[
\begin{split}
\frac{1}{|N|} \sum_{n \in N} \zeta_\tau(n)^2 & = \frac{1}{|N|} \sum_{n \in N} \zeta_\tau(n)
\overline{\zeta_\tau(n)}\\
& = \sum_{\sigma,\rho \in \widehat{N}} C_\tau(\sigma)C_\tau(\rho)  \frac{1}{|N|} \sum_{n \in N}
\chi_\sigma(n) \overline{\chi_\rho(n)}\\
& =_* \sum_{\sigma,\rho \in \widehat{N}} C_\tau(\sigma)C_\tau(\rho) \delta_{\sigma, \rho}\\
& = \vert\{\sigma \in \widehat{N}: C_\tau(\sigma) \neq 0\}\vert \\
& \equiv \vert\{\sigma \in \widehat{N}: \sigma^\tau \sim \sigma\}\vert
\end{split}
\]
where $=_*$ follows from the orthogonality relations for the characters, and therefore we get the equality between (a) and (c).
\end{proof}

\section{The twisted Frobenius-Schur theorem for a Gelfand pair} 
In this section we specialize the results of the previous section to the context of Gelfand pairs.

Let  $(G,K)$  be a Gelfand pair, $X$ and $x_0$ as in Section \ref{s:MFPR} and
\[
L(X) = \bigoplus_{\rho \in I}{ V_\rho}.
\]
the corresponding multiplicity  free decomposition.
By virtue of Frobenius  reciprocity, for each $\rho \in I$,  there exists a unique (modulo a complex factor of modulus $1$) unit vector $v \in V_\rho$ such that $\rho(k) v_\rho = v_\rho$ for all $k \in K$. The \emph{spherical function} associated with $v_\rho$  is the complex valued function $\phi_\rho$ on $G$ defined by 
\[
\phi_\rho(g) = \langle v_\rho, \rho(g)v_\rho\rangle_{V_\rho}
\]
for all $g \in G$. We observe that the spherical function $\phi_\rho$ is bi-$K$-invariant ($\phi_\rho(k_1gk_2) = \phi_\rho(g)$ for all $k_1,k_2 \in K$ and $g \in G$) and recall the following relations between $\phi_\rho$ and the corresponding character $\chi_\rho$ (cf. \cite[Exercise 9.5.8]{book}):
\begin{equation}\label{T15;diamond}
\phi_\rho(g) = \frac{1}{|K|}\sum_{k \in K}\overline{\chi_\rho(gk)} 
\end{equation}
and
\begin{equation}\label{T15;diamond2}
\chi_\rho(g) = \frac{d_\rho}{|G|}\sum_{h \in G}\overline{\phi_\rho(h^{-1}gh)} 
\end{equation}
for all $g \in G$, where $d_\rho = \dim V_\rho$.

\begin{theorem}[Twisted Frobenius-Schur for a Gelfand pair]
\label{t:TFSGP}
Let $\tau:G \to G$ be an involutory anti-automorphism.  Then we have
\begin{enumerate}[{\rm (1)}]
\item{For every $\rho \in I$
\[
\frac{d_\rho}{|G|}\sum_{g \in G}\phi_\rho(\tau(g)^{-1}g) = {C_\tau(\rho)}.
\]}
\item{For $x \in X$ we set $\zeta_\tau(x) = |\{g \in G: \tau(g)^{-1}gx_0 = x\}|$. Then we have
\[
\frac{1}{|G|}\sum_{x \in X}\zeta_\tau(x)^2 = |K|\sum_{\substack{\rho \in I:\\ \rho \sim \rho^\tau}}\frac{1}{d_\rho}.
\]}
\end{enumerate}
\end{theorem}
\begin{proof}
(1) From Theorem \ref{t1TFS} we obtain ($C_\tau(\rho)$ is real)
\[
\begin{split}
C_\tau(\rho) &  =  \frac{1}{|G|}\sum_{g \in G}\overline{\chi_\rho(\tau(g)^{-1}g)}\\
\mbox{(by   \eqref{T15;diamond2})}\ \ & = \frac{d_\rho}{|G|^2}\sum_{g,h \in G}\phi_\rho(h^{-1}\tau(g)^{-1}gh)\\
& = \frac{d_\rho}{|G|^2}\sum_{g,h \in G}\phi_\rho(\tau(\tau(h)gh)^{-1}\cdot \tau(h)gh)\\
\mbox{by setting $s=\tau(h)gh)$)}\ \ & = \frac{d_\rho}{|G|}\sum_{s \in G}\phi_\rho(\tau(s)^{-1}s).
\end{split}
\]

(2) From the previous fact (by setting $\varphi_\rho(x) = \phi_\rho(g)$ if $gx_0 = x$: note that this is well defined by virtue of the bi-$K$-invariance of $\phi_\rho$) we get
\[
C_\tau(\rho) =  \frac{d_\rho}{|G|}\sum_{x \in X}\zeta_\tau(x)\varphi_\rho(x).
\]
Then  the spherical Fourier inversion formula \cite[Equation (4.15)]{book} yields
\[
\zeta_\tau(x) = |K| \sum_{\rho\in I}\varphi_\rho(x)C_\tau(\rho).
\]
Therefore, from the orthogonality relations for spherical functions  \cite[Proposition 4.7.1]{book} we have
\[
\begin{split}
\frac{1}{|G|} \sum_{x\in X}\zeta_\tau(x)^2 & =  \frac{1}{|G|}\sum_{x\in X}\zeta_\tau(x)\overline{\zeta_\tau(x)}\\
& = \sum_{\rho \in I}C_\tau(\rho)^2\frac{|K|^2}{|G|}\cdot \frac{|X|}{d_\rho}\\
& = |K|\sum_{\substack{\rho \in I:\\ \rho^\tau \sim \rho}}\frac{1}{d_\rho}.
\end{split}
\]
\end{proof}

If $\tau(K) = K$ then $\tau$ induces an involution (that we keep denoting by $\tau$) on the set $K\backslash G/K$ of double cosets and therefore on the set $K\backslash X$ of $K$-orbits on $X$. This way, if $g\in G$ and $\Omega_{gx_0}$ is the $K$-orbit containing $gx_0$, then $\tau(\Omega_{gx_0})$ is the $K$-orbit containing  $\tau(g)x_0$.

\begin{theorem}[Twisted Frobenius-Schur for a Gelfand pair II]
Suppose that  $\tau(K)  = K$.  
\begin{enumerate}[{\rm (1)}]
\item{If  $\rho \in I$ then also $\rho^\tau \in I$ and the  number of $\rho \in I$  such that $\rho^\tau \sim \rho$ is equal to the number of $\tau$-invariant $K$-orbits on $X$.}
\item{If $\rho \in I$ and $\rho^\tau \sim \rho$ then $C_\tau(\rho) = 1$.}
\end{enumerate}
\end{theorem}
\begin{proof}
(1) Let us define $f_\rho \in V_\rho'$ by setting
\begin{equation}\label{agelos}
f_\rho(v) = \langle v, v_\rho\rangle_{V_\rho}
\end{equation}
for all  $v \in V_\rho$. Then we have 
\[
\phi_\rho(g) = f_\rho[\rho(g^{-1})v_\rho].
\]
Moreover, $\rho^\tau(k)f_\rho = f_\rho$ for all $k \in K$. Indeed, for all $k \in K$ and $v \in V_\rho$, we have:
\[
\begin{split}
[\rho^\tau(k)f_\rho]v & = f_\rho[\rho(\tau(k))v]\\
& = \langle \rho(\tau(k))v,v_\rho\rangle\\
& = \langle v,\rho(\tau(k))^{-1}v_\rho\rangle\\
\mbox{(since $\tau(K) = K$ and $v_\rho$ is $K$-invariant)}\ \  & = \langle  v , v_\rho\rangle.
\end{split}
\]
This shows that  $\rho^\tau \in I$ because $f_\rho$ is a non-trivial  $K$-invariant vector.
Then, we equip $V_\rho'$ with the scalar product given by duality so that
$\langle f, f_\rho\rangle_{V_\rho'} = f(v_\rho)$ for all $f \in V_\rho'$ and, recalling \eqref{agelos}, we  deduce
that the spherical function $\phi_{\rho^\tau}$ is given by
\[
\phi_{\rho^\tau}(g) = [\rho^{\tau}(g^{-1})f_\rho](v_\rho) = f_\rho[\rho(\tau(g))^{-1}v_\rho] = 
\langle v_\rho, \rho(\tau(g))v_\rho\rangle_{V_\rho} \equiv \phi_\rho(\tau(g))
\]
for all $g \in G$.
Thus,
\begin{equation}\label{T18stella}
\begin{split}
\frac{1}{|G|}\sum_{g \in G}\phi_\rho(g) \overline{\phi_\rho(\tau(g))} & = \frac{1}{|G|}
\sum_{g \in G}\phi_\rho(g)\overline{\phi_{\rho^\tau}(g)}\\
& = 
\begin{cases}
0 & \mbox{if } \rho \not\sim \rho^\tau\\
\frac{1}{d_\rho} &  \mbox{if } \rho \sim \rho^\tau.
\end{cases}
\end{split}
\end{equation}
On the other hand, by virtue of the dual orthogonality relations for spherical functions (follow from \cite[Proposition 4.7.1]{book}),
we get
\begin{equation}\label{T18stella2}
\sum_{\rho\in I}d_\rho\phi_\rho(g)\overline{\phi_\rho(\tau(g))} = 
\begin{cases}
\frac{|X|}{|\Omega_{gx_0}|} &  \mbox{if } \tau(g)x_0 \in \Omega_{gx_0}\\
0  & \mbox{otherwise.}
\end{cases}
\end{equation}
>From \eqref{T18stella} and  \eqref{T18stella2} we deduce 
\[
\begin{split}
|\{\rho\in I: \rho \sim \rho^\tau\}|  & =  \sum_{\rho \in I}\frac{1}{|G|}d_\rho\sum_{g \in G}\phi_\rho(g)\overline{\phi_\rho(\tau(g))}\\
& = \frac{1}{|G|}\sum_{\substack{\Omega \in K\backslash X:\\ \tau(\Omega) = \Omega}}\sum_{\substack{g \in G: \\gx_0 \in \Omega}}
\frac{|X|}{|\Omega|}\\
& = \sum_{\substack{\Omega \in K\backslash X:\\ \tau(\Omega) = \Omega}}\frac{1}{|G|}|K|\cdot |X|\\
& = |\{\Omega \in K\backslash X: \tau(\Omega) = \Omega\}|.
\end{split}
\]

(2)  Consider the permutation representation $(\lambda, L(X))$. By Theorem \ref{t:2MF} and Lemma \ref{l:1MF}.(3) we know that 
$C_\tau(\lambda) = m_1 =  |\{\Omega \in K\backslash X: \tau(\Omega) = \Omega\}|$. By Proposition \ref{PropCtausum} and Theorem \ref{thmCtC}.(1) we have that
$C_\tau(\lambda) = \sum_{\substack{\rho \in I:\\ \rho\sim \rho^\tau}}C_\tau(\rho)$.  Taking into account that, by the previous facts, $C_\tau(\lambda) = |\{\rho \in I: \rho \sim \rho^\tau\}|$, we conclude  that $C_\tau(\rho) = 1$ for all $\rho \in I$ such that $\rho \sim \rho^\tau$.
\end{proof}

\section{Examples}
In this section we review some examples related to our investigations. Note that all the examples discussed below
refer to involutive automorphisms of the given finite group $G$ while our treatement concerns involutive
\emph{anti}-automorphisms. Modulo the composition with the inverse map $\tau_{\text{\rm inv}}: g \mapsto g^{-1}$,
the two approaches are clearly equivalent.

Let us recall that, given a finite group $G$, a  \emph{Gelfand model}, briefly a \emph{model}, for $G$ (a notion introduced by I.N. Bernstein, I.M. Gelfand and S.I. Gelfand in \cite{bgg}) is a representation containing every irreducible representation with multiplicity one.  Models for the finite symmetric and general linear groups were described by A.A. Klyachko \cite{klyachko1, klyachko2}. 
In \cite{inglis3} Inglis, Richardson and Saxl presented a brief and elegant construction 
of an explicit \emph{involution} model for $S_n$ (the term ``involution'' refers to the fact that the model is obtained by summing up induced representations of subgroups which are centralizers of certain involutions). Their work was continued by Baddeley \cite{baddeley} who showed that if a finite group $H$ has an involution model, then the wreath product $H \wr S_n$ also has an involution model for any $n \in \NN$. As a byproduct, he
obtained involution models for Weyl groups of type $A_n$, $B_n$, $C_n$ and $D_{2n+1}$ for all $n \in \NN$. 
Note that the theorem of Frobenius-Schur (cf. Corollary \ref{c:auroraseimia} with $n=1_G$ and $\tau = \tau_\inv$) imposes an obstruction for a group $G$ to have an involution model: $G$ admits an involution model only if
$G$ has only real representations. 

In the spirit of the present paper, Bump and Ginzburg \cite{bump-ginzburg} considered \emph{generalized involution models}: these consist in replacing the involutions (resp. their centralizers) with twisted-involutions (resp. their twisted-centralizers) with respect to some involutive automorphism $\tau$ of the ambient group (thus a model for $G$ is a generalized involution model for $G$ with $\tau$ equal to the identity automorphism of $G$). 
In analogy with the standard involution models, we have the following obstruction: $G$ admits a generalized involution model (with respect to $\tau$) only if $C_\tau(\sigma)=1$ for all $\sigma \in \widehat{G}$.
We remark that the only abelian groups with involution models are $(\ZZ_2)^n$, $n \in \NN$, but every abelian group has generalized involution models. On the other hand, a Coxeter group has an involution model if and only if it has a
generalized involution model.
More recently, Marberg \cite{marberg} proved that if a finite group $H$ has a generalized involution model, then, in analogy to Baddley's main result, the wreath product $H \wr S_n$ also has a generalized involution model for any 
$n \in \NN$. As an application, it is shown that when $H$ is abelian, then $H \wr S_n$ has a model: when $H = \ZZ_r$, 
$r \in \NN$, this recovers a result previously obtained by Adin, Postnikov and Roichman \cite{Adin} (see below).

\begin{itemize}
\item
R.Gow \cite{gow1} considers the general linear group $G=\GL(n,k)$, where $k$ is a field, equipped with the involutory automorphism which sends each matrix $x\in G$ into its transposed inverse $(x^T)^{-1}$ (in our setting,
this corresponds to the involutory anti-automorphism $\tau$ which sends each $x\in G$ into its transposed $x^T$)
and the corresponding semi-direct product denoted $G^+$. It is first shown that every element of $G^+$ is a product of two involutions (Theorem 1) and therefore it is conjugate to its inverse, so that $G^+$ is ambivalent and all its irreducible representations are self-conjugate.  

Let now $k = \FF_q$ be the field with $q$ elements. Suppose $q$ is odd.
In Theorem 2, Gow shows that every irreducible representation of $G^+$ is indeed real. From this result the author deduces (Theorem 3) the formula
\begin{equation}
\label{e:gow}
\sum_{\sigma \in \widehat{G}} \chi_\sigma(g) = \zeta_\tau(g)
\end{equation}
for all $g \in G$. Comparing \eqref{e:gow} and \eqref{auroraseimia12} one deduces that $C_\tau(\sigma) = 1$ for all
$\sigma \in \widehat{G}$. Note that by taking $g=1_G$, the left hand side in \eqref{e:gow} gives the sum of dimensions
of all irreducible representations of $G$, equivalently the dimension of a model of $G$, while the right hand side
gives the number of symmetric matrices in $G$ (Theorem 4). 
One may remark that Theorem 2 can be derived from Theorem 4 using Theorem 1.
Moreover, Theorems 2 and 4 are both valid also for even $q$. In fact, Theorem 4 has been proved independently, for even as well as odd $q$, both by A. A. Klyachko \cite[Theorem 4.1]{klyachko2} and by I. G. Macdonald (unpublished manuscript).

\item
In \cite{gow2}, Gow considers the general linear group $G=\GL(n,\FF_{q^2})$, $q$ a prime power, and its subgroups
$U = \U(n,\FF_{q^2})$ (the unitary group of degree $n$ over $\FF_{q^2}$) and $M = \GL(n,\FF_{q})$. 
The \emph{Frobenius automorphism} $c \mapsto c^q$ of $\FF_{q^2}$ extends to an involutory automorphism $F$ of $G$ (leaving $U$ and $M$ invariant) by raising the entries of a matrix in $G$ to the $q$th power. Then also $F^*$, the composition of $F$ and the transposed inverse, is an involutory automorphism of $G$ (so that $U$ is the $G$-subgroup consisting of $F^*$-fixed elements).
By using these two automorphisms, it is then shown that $(G,U)$ (resp. $(G,M)$) is a Gelfand pair and that
the irreducible subrepresentations of $L(G/U)$ (resp. of $L(G/M)$) are precisely the $F$-fixed (resp. $F^*$-fixed) irreducible representations of $G$. We mention that the $F^*$-fixed representations of $G$ were used by Kawanaka
\cite{kaw} to give a parameterization of the irreducible representations of $U$.                                                  
\item
Inglis, Liebeck and Saxl in \cite{inglis1} consider the group $G_0 = \PSL(n,\FF_q)$ (the projective special linear group
of degree $n$ over $\FF_q$), with $n \geq 8$. Let $G$ be a group with \emph{socle} $G_0$, that is such that $G_0 \triangleleft G \leq \Aut(G_0)$. Then a description of all Gelfand pairs $(G,H)$ with $H$ maximal subgroup of $G$
not containing $G_0$ is given. Moreover, in \cite{inglis2} the authors finds a new model of the general linear group over a finite field (this construction can also be obtained from a result of Bannai, Kawanaka and Song \cite{bannai-KS} but the methods in \cite{inglis2} are independent of and different from theirs). 

\item
Vinroot \cite{vinroot} considers the group $G = \Sp(2n,\FF_q)$ equipped with the involutive automorphism
\[
g \mapsto \begin{pmatrix} -I_n & 0\\0 & I_n \end{pmatrix} g \begin{pmatrix} -I_n & 0\\0 & I_n \end{pmatrix}.
\]
Let us denote by $\tau$ the composition of the above automorphism and $\tau_{\text{\rm inv}}$, so that
\[
\tau(g) = \begin{pmatrix} -I_n & 0\\0 & I_n \end{pmatrix} g^{-1} \begin{pmatrix} -I_n & 0\\0 & I_n \end{pmatrix}.
\]

Observe that when $q\equiv 1$ (mod $4$) then $\tau$ is inner and every irreducible representation of $G$ is self-conjugate. Moreover, when $q\equiv 3$ (mod $4$) then $\tau$ is not inner, there exist irreducible representations of $G$ which are not self-conjugate, but $C_\tau(\sigma) = 1$ for all $\sigma \in \widehat{G}$ (cf. \cite[Theorem 1.3]{vinroot}). As a byproduct, from the analogous formula \eqref{e:gow} which holds in the present setting, Vinroot determines explcitly the dimension of any model of $\Sp(2n,\FF_q)$.

In \cite{vinroot2} Vinroot uses Klyachko's construction of a model for the irreducible complex representations of the finite general linear group $\GL(n,\FF_q)$ we alluded to above to establish, by determining the corresponding Frobenius-Schur number, whether a given irreducible self-conjugate representation of $\SL(n,\FF_q)$, the finite special linear group of degree $n$ over $\FF_q$, is real or quaternionic.

\item 
Adin, Postnikov and Roichman \cite{Adin} study Gelfand models for wreath products of the form $G=\ZZ_r \wr S_n$.
Any element $g$ of $G$ can be expressed uniquely as $g = \sigma v$, where $\sigma \in S_n$ and $v\in {\ZZ_r}^n$. 
Consider the map $\tau \colon G \to G$ given by $\tau(g) = \sigma(-v)$ for all $g = v\sigma\in G$.

An element $g \in G$ is said to be an {\it absolute square root} of another element $h \in G$ provided $g\tau(g) = h$.
Then the main result of this paper asserts that the value of the character associated to the Gelfand model of $G$ on $h$  equals the number $\zeta_\tau(h)$ of absolute square roots of $h$ (cf. \eqref{athos}). This generalizes a result concerning groups possessing only real characters, e.g. $S_n$ (in this case, absolute square roots coincide with square roots). 

\item
Bannai and Tanaka \cite{bannai} consider a finite group $G$, an automorphism $\sigma$ and the corresponding centralizer $K:= C_G(\sigma)$ in $G$, i.e. the subgroup consisting of all elements fixed by $\sigma$. It is well known and easy to see that for $g,h \in G$ the double cosets $KgK$ and $KhK$ are equal if and only if the elements $g\sigma(g)$ and
$h\sigma(h)$ are conjugate in $K$. Then the authors introduce the following condition:
 
\begin{itemize}
\item[$(\star)$] If the elements $g\sigma(g)$ and $h\sigma(h)$ are conjugate in $G$ then they are conjugate in $K$.
\end{itemize}
and showed (Proposition 1) that if $\sigma$ is an involution and condition $(\star)$ holds, then $(G,K)$ is a Gelfand pair. For instance, if $H$ is a finite group, $G = H \times H$, and $\sigma \colon G \to G$ is the {\it flip} defined by $\sigma(h_1,h_2) = (h_2,h_1)$, then $(\star)$ is satisfied and one recovers the well known fact that $(G,K)$ is
a Gelfand pair, where $K = C_\sigma(G)$ is $\widetilde{H}^2 = \{(h,h): h \in H\}$.

Moreover, they provided a list of other interesting examples where the above condition is satisfied.
In particular, when $G$ is the symmetric group $S_n$, with $n \geq 4$, their list exhausts all possible examples.
Other examples from the above mentioned list include some \emph{sporadic groups} as well as some linear groups including:

\begin{enumerate}[{\rm (i)}]
\item $G = \GL(n,\FF_{q^2})$, $K= \GL(n,\FF_{q})$;
\item $G = \GL(n,\FF_{q^2})$, $K= \GU(n,\FF_{q^2})$;                                                       
\item $G = \GL(2n,\FF_{q})$, $K= \Sp(2n,\FF_{q})$.
\end{enumerate}
Moreover they leave it as an open problem to determine whether condition $(\star)$ is satisfied in the case:
\begin{enumerate}
\item[{\rm (iv)}] $G = \GL(2n,\FF_{q})$, $K= \GL(n,\FF_{q^2})$.
\end{enumerate}

\end{itemize}

\section{Open problems and further comments}
Here below we indicate possible extensions and generalization of the results discussed by listing some open problems. 

\begin{comment}[\bf Multiplicity-free induced representations]
{\rm  In his work on induced representations, when looking for explicit criteria for multiplicity-freeness,
Mackey basically limited his investigation to permutation representations, that is, to representations obtained
by inducing the trivial representation of a subgroup. In this setting, the theory is rich and completely
understood. We recall the \emph{Gelfand-Garsia criterion} (Corollary \ref{cor:3MF}, cf. \cite[Example 4.3.2]{book}) 
for a symmetric Gelfand pair and the \emph{weak-Gelfand criterion} (Corollary \ref{cor:3MF-bis}, cf. \cite[Exercise 4.3.3]{book}). It is well known that there exist non-symmetric Gelfand pairs (with the cyclic groups and the alternating groups \cite[Example 4.8.3]{book}) as well as non-weakly-symmetric Gelfand pairs (\cite[Section 9.6]{book}). 

Let $G$ be a finite group and $K \leq G$ a subgroup.
Let also $\tau$ be an involutive antiautomorphism of $G$.

We notice that the sufficient condition in Corollary \ref{cor:3MF-bis}, namely Formula \eqref{quadratino}, is not a necessary condition for: (i) $(G,K)$ being a Gelfand pair and (ii) $C_\tau(\sigma) = 1$ for every irreducible representation 
$\sigma$ contained in the permutation representation $\lambda_\pi$. 

Bump and Ginzburg  \cite{bump-ginzburg} 
gave the following \emph{sufficient} condition (cf. \eqref{quadratino}) for $\Ind_K^G\theta$ being multiplicity-free, where $\theta$ is an irreducible $K$-representation: 
\begin{enumerate}[{\rm (i)}] 
\item $\tau(K) = K$; 
\item $\chi^\theta(\tau(g)) = \chi^\theta(g)$ for all $g \in G$;
\item $KgK$ is $\tau$-invariant for all $g \in G$; 
\item $\tau(s) = k_1sk_2$ for some $k_1, k_2 \in K$ such that $\chi^\theta(k_1)\chi^\theta(k_2)=1$, for all
$s \in S$,
\end{enumerate}
where $\chi^\theta$ denotes the character  of $\theta$ and $S$ is a suitable complete set of representatives for the
double cosets of $K$ in $G$.

It would be interesting to find an analogue of the Mackey-Gelfand criterion (cf. Theorem \ref{t:3MF}) along the lines of \cite{bump-ginzburg} as well as to find an example of a $K$-representation $\theta \neq \iota_K$ such that $\Ind_K^G\theta$ is multiplicity-free but does not satisfy the Bump and Ginzburg criterion above.   

The first step towards the first part of the above problem, should be to examine the case $\dim \theta = 1$. 
In \cite{AM2} we consider a \emph{Hecke algebra} ${\mathcal H}(G,K,\theta)$ and presented a sufficient condition (a Garsia-type criterion on ${\mathcal H}(G,K,\theta)$) for $\Ind_K^G\theta$ being multiplicity-free and we illustrate it with the following example. Let $G = \GL(2,\FF_q)$ denote the group of invertible $2 \times 2$ matrices with coefficients in $\FF_q$, the Galois field with $q$ elements, and $K = \left\{\begin{pmatrix}  1 & b\\ 0& 1 \end{pmatrix}: b \in \FF_q\right\}$ the subgroup of {\it unipotent} matrices. Then for every non-trivial character $\chi$ of the (abelian) subgroup $K$ the induced representation $\Ind_K^G \chi$ is multiplicity-free. This is a particular case of the \emph{Gelfand-Graev representation} of a simple group of Lie type obtained by inducing a non-trivial character of the maximal unipotent subgroup.}
\end{comment}

Now we list two open problems that, together with \ref{multfrresub} and \ref{opth74}, suggest that the Mackey-Wigner theory should be a particular case of a more general theory. 

\begin{problem}[\bf Harmonic analysis and tensor products]
{\rm Let $G$ be a finite group and $\tau$ an involutive anti-automorphism of $G$.
Suppose $G$ is $\tau$-simply reducible. What is the relation between the spherical Fourier analysis on the homogeneous space $L(G^3/\widetilde{G}^3)$ (see \cite[Section 4]{book}) and the decomposition of the tensor products?}
\end{problem}

\begin{problem}[\bf Decomposition of tensor products]
{\rm
Suppose $(G^3,\widetilde{G}^3)$ is not a Gelfand pair. Is it possible to find rules that relate the decomposition of $L(G^3/\widetilde{G}^3)$ into irreducible representations with the decomposition of tensor products of irreducible $G$-representations? A possible strategy could be to apply, in this context, the analysis developed in \cite{st1,st2} for permutation representations that decompose with multiplicity. Moreover, a possible application should be to shed light to one of the major open problem in the representation theory of the symmetric group, namely the decomposition of the tensor product of two irreducible representations (usually called {\em Kronecker products}). See \cite[Section 2.9]{JK} for an introduction, \cite{GR} as a classical reference, and \cite{GW} as a recent interesting paper. 
Explicit decompositions of tensor products are also useful in the determination of the lower bound for the rate of convergence to the stationary distibution for diffusion processes on finite groups; see \cite{Dbook} and   \cite[Section 10.7]{book}.}
\end{problem}

\begin{problem}[\bf Characterization of simply reducible groups]
{\rm
The major open problem in the theory of simply (or $\tau$-simply) reducible groups is to give a nice and useful characterization of these groups. This was stated as an open problem in the famous Kourovka notebook \cite{KO}. 
A great advance on this problem is in the recent paper \cite{KC} where the authors show that all $\tau$-simply reducible groups are soluble (this also was an open problem in \cite[Problem 11.94]{KO}, 
posed by Strunkov (see also \cite{strunkov}).}
\end{problem}

\begin{comment}[McKay correspondence] {\rm 
Simple reducibility is also relevant to the McKay correspondence which we now describe. 

Let $G$ be a finite group. Given a representation $\sigma$ of $G$, the \emph{McKay quiver} associated with
$\sigma$ is the directed multi-graph defined as follows: the vertex set is $\widehat{G}$ and, given $\rho_1, \rho_2 \in\widehat{G}$, there are $m_{\rho_1,\rho_2}^\sigma$ directed edges from $\rho_1$ to $\rho_2$, where 
$m_{\rho_1,\rho_2}^\sigma$ is the multiplicity of $\rho_2$ in $\sigma \otimes \rho_1$.

Let now $\pi \colon \SU(2) \to \SO(3)$ denote the standard double cover and note that the only element of even order in 
$\SU(2)$ is the generator $-1$ of the kernel of $\pi$. Therefore, any finite subgroup of $\SU(2)$ either has even order (and is the preimage of some finite subgroup of $\SO(3)$) or has odd
order (and is isomorphic to a finite subgroup of $\SO(3)$ of odd order, hence a cyclic group). Now,
the finite subgroups of $\SO(3)$ are: the cyclic groups $\ZZ/n\ZZ$, the dihedral groups $D_{2n}$, 
the tetrahedral group $T$ (i.e. the alternating group $A_4$), the octahedral group $O$ (i.e. the
symmetric group $S_4$), and the icosahedral rotation $I$ (i.e. the alternating group $A_5$). It follows that
the finite subgroups of $\SU(2)$ are: the cyclic groups, the binary dihedral groups $BD_{2n}= \pi^{-1}(D_{2n})$, the binary tetrahedral group $BT = \pi^{-1}(T)$, the binary octahedral group $BO= \pi^{-1}(O)$, and the binary icosahedral rotation group $BI= \pi^{-1}(D)$.

Let $G$ be any finite subgroup of $\SU(2)$ as above and let $\sigma$ denote the faithful representation of $G$ obtained from the
embedding $G \hookrightarrow \SU(2)$. Then, one can show that the associated McKay quiver is connected and has
no self-loops ($m_{\rho,\rho}^\sigma = 0$ for all $\rho\in \widehat{G}$). Moreover, simple reducibility of $G$ 
implies that the McKay quiver is a simple and undirected graph (i.e. $m_{\rho_1,\rho_2}^\sigma \in \{0,1\}$ and
$m_{\rho_1,\rho_2}^\sigma = m_{\rho_2,\rho_1}^\sigma$  for all $\rho_1, \rho_2 \in \widehat{G}$).

The \emph{McKay correspondence}, named after John McKay \cite{mckay,etingof}, then states that the construction of
McKay quivers yields a bijection between the non-trivial finite subgroups of $\SU(2)$ and the affine simply laced Dynkin diagrams (which appear in the A-D-E classification of simple Lie Algebras). For an overview of the correspondence and other mathematical structures which appear in connection with solvable models (e.g. the ice-type, Potts, and spin models) in two-dimensional statistical physics, see \cite[Section 2]{jones}.
}\end{comment}

\begin{comment}[Simple phase groups]
{\rm Let $G$ be a finite group and let $(\sigma_i,V_i)$, $i=1,2,\ldots, n$ denote a list of all pairwise inequivalent irreducible representations of $G$.
Let $d_i = \dim(V_i)$ and fix an orthonormal basis $\{v^i_s: s=1,2,\ldots,d_i\}$ in $V_i$, for all $i=1,2,\ldots,n$.
By multiplicity freeness, given $1 \leq i, j \leq n$ we can find $1 \leq i_1 < i_2 < \cdots < i_k \leq n$ such that
\begin{equation}
\label{e:decomp-wigner}
\sigma_i \otimes \sigma_{j} \sim \bigoplus_{t=1}^k \sigma_{i_k}.
\end{equation}
Consider the vector space $V_i \otimes V_j$ and let $T \in \Hom_G(\bigoplus_{t=1}^k V_{i_k}, V_i \otimes V_j)$ be an unitary intertwiner (cf. \eqref{e:decomp-wigner}).
There are two natural orthonormal bases in $V_i \otimes V_j$, namely $\{v^i_s\otimes v^j_t: s=1,2,\ldots,d_i, t=1,2,\ldots,d_j\}$and $\{T(v^{i_\ell}_{u_\ell}): u_\ell=1,2,\ldots, d_{i_\ell}, \ell=1,2,\ldots,k\}$. Then 
we can express
\[
T(v^{i_\ell}_{u_\ell}) = \sum_{i=1}^n\sum_{j=1}^n\sum_{s=1}^{d_i}\sum_{t=1}^{d_j}C^{i,j,i_\ell}_{s,t,u_\ell} v^i_s\otimes v^j_t
\]
where the complex numbers
\begin{equation}
\label{e:wigner-coeff}
C^{i,j,i_\ell}_{s,t,u_\ell}
\end{equation} 
called the \emph{Clebsh-Gordan coefficients} of \emph{Wigner coefficients} constitute the unitary matrix of change of base.   
As Wigner \cite{wignerbook} showed (see also \cite[Chapter 5]{hamermesh}), the Clebsh-Gordan coefficients can be chosen in such a way that their absolute values are invariant under every permutation of the $i,j,i_\ell$'s and the corresponding $s,t,u_\ell$'s: in other words they change only by a multiplicative phase factor.
If one drops the property of ambivalence no essential new difficulties arise in
the definition and in the symmetry relations of Clebsh-Gordan coefficients. However, if
the multiplicity free condition is dropped, then a multiplicity index enters the
Clebsh-Gordan coefficients. Derome has shown \cite{derome} that these multiplicity Clebsh-Gordan coefficients
are invariant under permutations in the above sense if and only if
\begin{equation}
\label{e:sf}
\sum_{g \in G} \chi^\sigma(g^3) = \sum_{g \in G} \chi^\sigma(g)^3
\end{equation}
for all $\sigma \in \widehat{G}$.
Groups for which \eqref{e:sf} holds are called \emph{simple phase groups} (see \cite{derome, BK1, BK2, king, vZdV, EK}).
In \cite{vZdV} van Zanten and de Vries  derive several Mackey-Wigner type criteria for the existence of a real representation and then derive analogues of some of them, giving criteria for $G$ to fail to be a simple phase group.
It would be interesting to investigate twisted versions (in terms of an involutive (anti)automorphism $\tau$ of the group $G$) of their results.} 
\end{comment}

\begin{problem}[\bf Multiplicity-free subgroups]\label{multfrresub}
{\rm Let $G$ be a finite group and $H \leq G$ a subgroup.
We say that $H$ is a {\em multiplicity-free subgroup} of $G$ when $\Res^G_H\rho$ decomposes without multiplicity for all $\rho\in\widehat{G}$. See \cite{wigner2}, our book \cite{book2} for its relations with the theory of Gelfand-Tsetlin basis and the Okounkov-Vershik approach to the representation theory of the symmetric group, and \cite{st2} for the not multiplicity-free case. 
In particular, in \cite[Theorem 2.1.10]{book2} we presented a general criterion for the subgroup $H$ being multiplicity-free in terms of commutativity of the algebra ${\mathcal C}(G,H)$ of $H$-conjugacy invariant functions on $G$ and
of the Gelfand pair $(G \times H, \widetilde{H})$. Also, in \cite[Proposition 2.1.12]{book2} we presented the
following sufficient condition: for all $g \in G$ there exists $h \in H$ such that $h^{-1}gh = g^{-1}$.
We then used this criterion to show the well known fact that $S_{n-1}$ is a multiplicity-free subgroup of $S_n$,
the symmetric group of degree $n$ (cf. \cite[Theorem 3.2.1 and Corollary 3.2.2]{book2}).

One of the key facts of the theory is that $H$ is multiplicity-free if and only if $(G\times H,\widetilde{H})$ is a Gelfand pair, where $\widetilde{H}={(h,h):h\in H}$. This is a generalization of \eqref{penali}. Then it should be interesting to examine pairs like $(G\times G\times H, \widetilde{H}^3)$ or $(G\times H\times H, \widetilde{H}^3)$ and their relations with the representation theory of $G$ and $H$.}
\end{problem}

\begin{problem}\label{opth74}
{\rm
Theorem \ref{tulliosenzasperanza} gives a representation theoretical interpretation of the purely group theoretical quantities in Corollary \ref{c:8SR2}. Is there a representation theoretical interpretation of the more general quantities in Theorem \ref{t:7SR2}?}
\end{problem}

\noindent
{\bf Acknowledgments.} We express our deepest gratitude to Hajime Tanaka for interesting discussions.
We also warmly thank the referee for his remarks and suggestions (expecially for providing us with several
useful and important references) which undoubtedly contributed to improve our paper.


 \end{document}


\maketitle 

\begin{abstract}
A pair $(G,K)$ of a finite group $G$ and a subgroup $K$ of $G$ is called a Gelfand pair, if the permutation character $\pi =(1_K)^G$ (of the action of $G$ on the cosets $G/K$) is multiplicity-free.
This condition is equivalent to the condition that the Hecke algebra $\mathcal{H}(K\backslash G/K)$ is commutative, equivalently (the Bose-Mesner algebra of) the associated association scheme obtained from the action of $G$ on the cosets $G/K$ is commutative.
Among general commutative association schemes, those which are obtained from Gelfand pairs are very special ones.
To study what Gelfand pairs do exist, and/or  what are their associated zonal spherical functions (or equivalently what are the character tables of the associated association schemes or Hecke algebras) are very fundamental important problems in this research area.
While, there are many open problems left yet to be answered.

 We pay close attention on a special condition related to Gelfand pairs.
Namely, we call a finite group $G$ and its automorphism $\sigma$ satisfy  Condition ($\bigstar$) if the following condition is satisfied:
if for $x,y\in G$, $x\cdot x^{-\sigma}$ and $y\cdot y^{-\sigma}$ are conjugate in $G$, then they are conjugate in $K=C_G(\sigma)$.
The main purpose of the note was to study the meanings of this condition, as well as showing many examples of $G$ and $\sigma$ which do (or do not) satisfy Condition ($\bigstar$).
\end{abstract} 

\section{Condition ($\bigstar$)}
Let $G$ be a finite group, $\sigma$ an automorphism of $G$ (which may be an inner or outer automorphsm), and $K=C_G(\sigma)$.

\begin{fact}[Well-known]{\rm
For $x,y\in G$, we have $KxK=KyK$ if and only if $x\cdot x^{-\sigma}$ and $y\cdot y^{-\sigma}$ are conjugate in $K$.
(Proof is straightforward.)}
\end{fact}

\begin{definition}{\rm 
We say Condition ($\bigstar$) is satisfied, if the following condition holds: \\ 
If $x\cdot x^{-\sigma}$ and $y\cdot y^{-\sigma}$ are conjugate in $G,$ then they are conjugate in $K$.}
\end{definition}

Let $\Omega=\{x\cdot x^{-\sigma}|x\in G\}$. Then, obviously, 
$|\Omega/\text{conjugacy in}\; K|\ge |\Omega/\text{conjugacy in}\; G|$. 
The equality $|\Omega/\text{conjugacy in}\; K|= |\Omega/\text{conjugacy in}\; G|$ holds, if and only if Condition ($\bigstar$) is satisfied. 

\begin{proposition}
Let $G$ and $\sigma$ satisfy Condition ($\bigstar$).
If $\sigma^2=1$ (i.e., if $\sigma$ is an involution), then $(G,K)$ becomes a Gelfand pair.
\end{proposition}
\begin{proof}
$x\cdot x^{-\sigma}=x\cdot x^{-\sigma}\cdot x\cdot x^{-1}=x\cdot x^{-\sigma}\cdot (x^{-\sigma})^{-\sigma}\cdot x^{-1}$, because $\sigma^2=1$.
Therefore, by Condition ($\bigstar$), we have 
$x\cdot x^{-\sigma}$ and $x^{-\sigma}\cdot (x^{-\sigma})^{-\sigma}$ are conjugate in $K$.
Hence, by the above fact, we have $KxK=Kx^{-\sigma}K$.
Thus since the following Gelfand's criterion is satisfied, $(G,K)$ becomes a Gelfand pair.

Criterion of Gelfand: Let $\sigma$ be an automorphism of $G$ and let $Kx^{-1}K=Kx^{\sigma}K,\forall x\in G$.
Then $(G,K)$ is a Gelfand pair.
(Cf. Terras \cite[pp.~307--308]{Terras}.)
\end{proof}

\begin{remark}{\rm
It seems to be an open problem whether there is an example of $G$ and $\sigma$ satisfying Condition ($\bigstar$), but $(G,K)$ is not a Gelfand pair.}
\end{remark}

\section{Examples}
Here we consider some examples which satisfy Condition ($\bigstar$).

\begin{example}{\rm
Let $G=S_{2n}$, the symmetric group on $2n$ letters.
Let $\sigma=(1,2)(3,4)\dots (2n-1,2n)$ be a fixed point free involutive element of $G$, and let $K=C_G(\sigma)$.
(Then $K$ is isomorphic to the Weyl group of type $B_n$.)
Then $X=G/K$ is identified with the set of fixed point free involutions of $G=S_{2n}$.

Note that, $(1_K)^G$ is multiplicity-free, and $(G,K)$ is a Gelfand pair.
Moreover, $||(1_K)^G||=p(n)$ and $(1_K)^G=\bigoplus \chi_D$, where $D$ runs all the even Young diagram $D$ of size $2n$, where a Young diagram $D=(n_1,n_2,\dots)$ is called even if $n_i\equiv 0 \pmod{2}, \forall i$, and $\chi_D$ is the irreducible character of $S_{2n}$ corresponding to the Young diagram (=partition) $D$.
Then $\Omega= \{xy|x,y\in X\}$ consists of those elements of $S_{2n}$ whose cycle decomposition corresponds to the Young diagram (partition) $D$ with $^tD$ being even.
That is, all the cycles in the cycle decomposition are paired into two cycles of equal length.
This is proved directly and in a very elementary way.}
\end{example}

\begin{example}{\rm 
Let $G=S_{2n+1}$, the symmetric group on $2n+1$ letters.
Let $\sigma=(1,2)(3,4)\dots (2n-1,2n)$ be a one-fixed-point involutive element of $G$, and let $K=C_G(\sigma)$.
(Then $K$ is isomorphic to the Weyl group of type $B_n$.)
Then $X=G/K$ is identified with the set of one-fixed-point involutions of $G=S_{2n}$.

Note that, $(1_K)^G$ is multiplicity-free, and $(G,K)$ is a Gelfand pair.
Moreover, $(1_K)^G=\bigoplus \chi_D$, where $D$ runs all Young diagram $D$ of size $2n+1$, where a Young diagram $D=(n_1,n_2,\dots)$ satisfy $n_i\equiv 0 \pmod{2}, \forall i,$ except exactly one $i$.
Then $\Omega= \{xy|x,y\in X\}$ consists of those elements of $S_{2n+1}$ whose cycle decomposition corresponds to the Young diagram (partition) $D$ with exactly one odd cycle which are not paired. 
This is proved directly and in a very elementary way.}
\end{example}

\begin{example}{\rm 
Let $G=S_n$, the symmetric group on $n$ letters.
$\sigma=(1,2)$, and $K=C_G(\sigma)$.
Then $\Omega$ consists of the identity element, elements of two 2-cycles and elements of a 3-cycle.
Obviously Condition ($\bigstar$) is satisfied, and $(G,K)$ is a Gelfand pair.
Moreover, as is well known, $(1_K)^G=\chi_{(n)}+\chi_{(n-1,1)}+\chi_{(n-2,2)}$.}
\end{example}

It can be proved that, for $n\ge 4$, the above three Examples are the only examples for which $G$ is a symmetric group and $\sigma$ satisfies Condition ($\bigstar$).

The following is a general example.

\begin{example}{\rm 
Let $H$ be a finite group, and let $G=H\times H, \sigma :(x,y)\longmapsto (y,x)$.
Then Condition ($\bigstar$) is satisfied, and $K=\{(x,x)|x\in H\}\cong H$, and $(G,K)$ is a Gelfand pair.}
\end{example}

There are many more examples of $G$ and $\sigma$.
Some examples are as follows.
We denote $G$ and subgroup $K=C_G(\sigma)$.

\renewcommand{\labelenumi}{(\theenumi)}
\begin{enumerate}
\item $G=PGL(2,q), K=D_{2(q-1)}$ and $G=PGL(2,q), K=D_{2(q+1)}$,
\item $G=U_n(q), K=U_{n-1}(q)$,
\item $G=O_{2m}^{\pm}(q), K=O_{2m-1}(q)$ \ (classical groups acting on nonisotropic points),
\end{enumerate}
etc. etc.
There are many other examples of such Gelfand pairs satisfying Condition ($\bigstar$) for appropriate $\sigma$, for example, see Inglis' thesis \cite{Inglis}.

Moreover, there are many sporadic examples.

\begin{itemize}
\item $G=M$ (Monster simple group), $K=C_G(2A)\cong 2\cdot BM$ \ (then the decomposition of $\pi=(1_K)^G$ is related to the extended Dynkin diagram of type $E_8$),
\item $G=Fi_{24}'=F_{3+}, K=C_G(2C)$ where $\sigma$ is an outer automorphism,
\item $G=Fi_{23}, K=C_G(2A)$,
\item $G=Fi_{22}, K=C_G(2A)$,
\end{itemize}
etc. etc.
There are many other examples of $3$-transposition groups.
We did check many examples of $G$ and $\sigma$ do satisfy Condition ($\bigstar$) by checking \textbf{ATLAS} \cite{ATLAS}.
Also, we got many examples $G$ and $\sigma$ not satisfying Condition ($\bigstar$), by looking through \cite{ATLAS}.
(The authors thank Akihiro Munemasa for checking many examples by using MAGMA.)
But, we will not go further on these examples.
Instead, we will consider the following four infinite families listed by Inglis \cite{Inglis} for which $G=GL(n,q)$:
\renewcommand{\theenumi}{\roman{enumi}}
\begin{enumerate}
\item $G=GL(n,q^2), K=GL(n,q)$,
\item $G=GL(n,q^2), K=GU(n,q^2)$,
\item $G=GL(2n,q), K=Sp(2n,q)$ and
\item $G=GL(2n,q), K=GL(n,q^2).$
\end{enumerate}
We can see that the first three cases do satisfy Condition ($\bigstar$).
(Cf. Inglis-Liebeck-Saxl \cite{I-L-S}, Gow \cite{Gow}, Bannai-Kawanaka-Song \cite{B-Ka-S}, etc.)
We conjectured that Condition ($\bigstar$) is also satisfied for the last case (iv).
Also, we conjectured those elements of $G$ which appear in $\Omega=\{x\cdot x^{-\sigma}|x\in G\}$ are closely related to the decomposition of the permutation character $(1_{GL(n,q^2)})^{GL(2n,q)}$ which were studied in our previous paper \cite{B-T} and Henderson \cite{Henderson03}, etc. However, the situation is not so simple as we expected.
We discuss this situation in the next section.

\section{Discussions on the case $G=GL(2n,q), K=GL(n,q^2)$}

When $n=1$, it is very easy to see that Condition ($\bigstar$) is satisfied.
On the other hand, we will see in this section that when $n=2$ there is an automorphism $\sigma$ of $G=GL(4,q)$ with $C_G(\sigma)=K=GL(2,q^2)$ for which Condition ($\bigstar$) fails (but only \emph{very} slightly).
We note that, unlike (i)--(iii) above, $\sigma$ is not an involution.

Let $\Phi$ be the set of monic irreducible polynomials in $\mathbb{F}_q[t]$ other than $t$, and let $\mathcal{P}$ be the set of all partitions.
Then, it is well known that the irreducible characters of $GL(n,q)$ can be parametrized by the partition-valued functions $\boldsymbol{\mu}:\Phi\longrightarrow\mathcal{P}$ such that
\[
	||\boldsymbol{\mu}||=\sum_{f\in\Phi}\boldsymbol{\mu}(f)d(f)=n,
\]
where $d(f)$ is the degree of $f$.
(Cf. Green \cite{Green1955TAMS} and Macdonald \cite{Macdonald1995B}; our parametrization follows that of \cite{Macdonald1995B}.)

The decomposition of the permutation character $(1_{GL(n,q^2)})^{GL(2n,q)}$ is given as follows:
\begin{theorem}[\cite{B-T,Henderson03}]
\textup{(i)} If $q$ is odd, then we have $(1_{GL(n,q^2)})^{GL(2n,q)}=\sum \chi_{\boldsymbol{\mu}}$, summed over $\boldsymbol{\mu}$ such that $||\boldsymbol{\mu}||=2n$, $\tilde{\boldsymbol{\mu}}=\boldsymbol{\mu}$, and both ${}^t\boldsymbol{\mu}(t+1)$ and $\boldsymbol{\mu}(t-1)$ are even.

\hspace*{1em} \textup{(ii)} If $q$ is even, then we have $(1_{GL(n,q^2)})^{GL(2n,q)}=\sum \chi_{\boldsymbol{\mu}}$, summed over $\boldsymbol{\mu}$ such that $||\boldsymbol{\mu}||=2n$, $\tilde{\boldsymbol{\mu}}=\boldsymbol{\mu}$, and ${}^t\boldsymbol{\mu}(t+1)$ is even.

\hspace*{1em} \textup{(iii)} In either case, the generating function for the rank (i.e., the number of the irreducible constituents of the permutation character $(1_{GL(n,q^2)})^{GL(2n,q)}$) is given by
\[
        \sum_{n\ge 0}\rank(GL(2n,q)/GL(n,q^2))t^{2n}=\prod_{r\ge 1}(1-qt^{2r})^{-1}
\]
with the understanding that $\rank(GL(2\cdot 0,q)/GL(0,q^2))=1$.
In particular we have
\[
        \rank(GL(2n,q)/GL(n,q^2)=\sum q^{l(\lambda)}
\]
summed over all partitions $\lambda$ with $|\lambda|=n$.
\end{theorem}

From now on we consider the case $n=2$.
Let $\alpha\in\mathbb{F}_{q^2}-\mathbb{F}_q$ and $g$ an element of $G=GL(4,q)$ with Jordan canonical form $\left(\begin{smallmatrix} \alpha &&& \\ & \alpha && \\ && \alpha^q & \\ &&& \alpha^q \end{smallmatrix}\right)$.
Then, the fixed-point subgroup $K=C_G(\sigma)$ of the automorphism $\sigma:x\longmapsto gxg^{-1}$ is isomorphic to $GL(2,q^2)$.
Using a famous formula for an explicit expression of the structure constants of the group algebra in terms of character values, we determined the conjugacy classes of $G$ that intersect $\Omega=\{x\cdot x^{-\sigma}|x\in G\}$, in the case $\alpha=\tau^{q-1}$ where $\tau$ is a primitive element of $\mathbb{F}_{q^2}$.
(For other $\alpha$, the computation could be much more complicated.)
The result is given in the following table.
\begin{center}
\scriptsize
\begin{longtable}{|c|c|c|}
\caption{\textit{The Orders of the ``Almost'' Double Cosets $/|GL(2,q^2)|$}} \\
\hline
\hline 
&& \\[-1.7mm]
Type & Conditon & Size \\[1.5mm]
\hline
\endhead
$\begin{array}{c}\\[-2mm]\begin{pmatrix}
	1&&& \\ &1&& \\ &&1& \\ &&&1
\end{pmatrix}\\ \\[-2mm]\end{array}$
&& $1$ \\
\hline%
$\begin{array}{c}\\[-2mm]\begin{pmatrix}
	1&1&& \\ &1&& \\ &&1&1 \\ &&&1 
\end{pmatrix}\\ \\[-2mm]\end{array}$
&& $(q-1)(q+1)(q^2+1)$ \\
\hline
$\begin{array}{c}\\[-2mm]\begin{pmatrix}
	-1&1&& \\ &-1&& \\ &&-1&1 \\ &&&-1
\end{pmatrix}\\ \\[-2mm]\end{array}$
&& $q(q+1)(q^2+1)$ \\
\hline%
$\begin{array}{c}\\[-2mm]\begin{pmatrix}
	-1&1&& \\ &-1&1& \\ &&-1&1 \\ &&&-1
\end{pmatrix}\\ \\[-2mm]\end{array}$
&& $q(q-1)(q+1)^2(q^2+1)$ \\
\hline%
$\begin{array}{c}\\[-2mm]\begin{pmatrix}
	\rho^a&&& \\ &\rho^a&& \\ &&\rho^{-a}& \\ &&&\rho^{-a}
\end{pmatrix}\\ \\[-2mm]\end{array}$
& $a\not\equiv0,\frac{q-1}{2}\pmod{q-1}$ & $q(q+1)(q^2+1)$ \\
\hline%
$\begin{array}{c}\\[-2mm]\begin{pmatrix}
	-1&1&& \\ &-1&& \\ &&1& \\ &&&1
\end{pmatrix}\\ \\[-2mm]\end{array}$
&& $q^2(q+1)(q^2+1)$ \\
\hline%
$\begin{array}{c}\\[-2mm]\begin{pmatrix}
	\rho^a&1&& \\ &\rho^a&& \\ &&\rho^{-a}&1 \\ &&&\rho^{-a}
\end{pmatrix}\\ \\[-2mm]\end{array}$
& $a\not\equiv0,\frac{q-1}{2}\pmod{q-1}$ & $q(q-1)(q+1)^2(q^2+1)$ \\
\hline%
$\begin{array}{c}\\[-2mm]\begin{pmatrix}
	1&&& \\ &1&& \\ &&\rho^a& \\ &&&\rho^{-a}
\end{pmatrix}\\ \\[-2mm]\end{array}$
& $a\not\equiv0,\frac{q-1}{2}\pmod{q-1}$ & $q^2(q+1)(q^2+1)$ \\
\hline%
$\begin{array}{c}\\[-2mm]\begin{pmatrix}
	-1&1&& \\ &-1&& \\ &&\rho^a& \\ &&&\rho^{-a}
\end{pmatrix}\\ \\[-2mm]\end{array}$
& $a\not\equiv0,\frac{q-1}{2}\pmod{q-1}$ & $q^2(q+1)^2(q^2+1)$ \\
\hline%
$\begin{array}{c}\\[-2mm]\begin{pmatrix}
	\rho^a&&& \\ &\rho^{-a}&& \\ &&\rho^{a'}& \\ &&&\rho^{-a'}
\end{pmatrix}\\ \\[-2mm]\end{array}$
& $\begin{array}{c} a,a'\not\equiv0,\frac{q-1}{2}\pmod{q-1}, \\[1.5mm] a\not\equiv a',-a'\pmod{q-1}\end{array}$ & $q^2(q+1)^2(q^2+1)$ \\
\hline
$\begin{array}{c}\\[-2mm]\begin{pmatrix}
	1&&& \\ &1&& \\ &&\tau^{2(q-1)}& \\ &&&\tau^{-2(q-1)}
\end{pmatrix}\\ \\[-2mm]\end{array}$
&& $q^2(q^2+1)$ \\
\hline%
$\begin{array}{c}\\[-2mm]\begin{pmatrix}
	1&&& \\ &1&& \\ &&\tau^b& \\ &&&\tau^{bq}
\end{pmatrix}\\ \\[-2mm]\end{array}$
& $\begin{array}{c} q-1\mid b, \\[1.5mm] b\not\equiv 0,\frac{q^2-1}{2},\pm 2(q-1) \pmod{q^2-1}\end{array}$ & $q^2(q+1)(q^2+1)$ \\
\hline%
$\begin{array}{c}\\[-2mm]\begin{pmatrix}
	-1&1&& \\ &-1&& \\ &&\tau^{2(q-1)}& \\ &&&\tau^{-2(q-1)}
\end{pmatrix}\\ \\[-2mm]\end{array}$
&& $q^2(q+1)(q^2+1)$ \\
\hline%
$\begin{array}{c}\\[-2mm]\begin{pmatrix}
	-1&1&& \\ &-1&& \\ &&\tau^b& \\ &&&\tau^{bq}
\end{pmatrix}\\ \\[-2mm]\end{array}$
& $\begin{array}{c} q-1\mid b, \\[1.5mm] b\not\equiv 0,\frac{q^2-1}{2},\pm 2(q-1) \pmod{q^2-1}\end{array}$ & $q^2(q+1)^2(q^2+1)$ \\
\hline%
$\begin{array}{c}\\[-2mm]\begin{pmatrix}
	\rho^a&&& \\ &\rho^{-a}&& \\ &&\tau^{2(q-1)}& \\ &&&\tau^{-2(q-1)}
\end{pmatrix}\\ \\[-2mm]\end{array}$
& $a\not\equiv0,\frac{q-1}{2}\pmod{q-1}$ & $q^2(q+1)(q^2+1)$ \\
\hline%
$\begin{array}{c}\\[-2mm]\begin{pmatrix}
	\rho^a&&& \\ &\rho^{-a}&& \\ &&\tau^b& \\ &&&\tau^{bq}
\end{pmatrix}\\ \\[-2mm]\end{array}$
& $\begin{array}{c} 
		a\not\equiv0,\frac{q-1}{2}\pmod{q-1},\;q-1\mid b, \\[1.5mm]
		b\not\equiv 0,\frac{q^2-1}{2},\pm 2(q-1) \pmod{q^2-1}
	\end{array}$ & $q^2(q+1)^2(q^2+1)$ \\
\hline%
$\begin{array}{c}\\[-2mm]\begin{pmatrix}
	\tau^{2(q-1)}&&& \\ &\tau^{-2(q-1)}&& \\ &&\tau^{2(q-1)}& \\ &&&\tau^{-2(q-1)}
\end{pmatrix}\\ \\[-2mm]\end{array}$
&& $q^4$ \\
\hline%
$\begin{array}{c}\\[-2mm]\begin{pmatrix}
	\tau^b&&& \\ &\tau^{bq}&& \\ &&\tau^b& \\ &&&\tau^{bq}
\end{pmatrix}\\ \\[-2mm]\end{array}$
& $\begin{array}{c} q-1\mid b, \\[1.5mm] b\not\equiv 0,\frac{q^2-1}{2},\pm 2(q-1) \pmod{q^2-1}\end{array}$ & $q(q+1)(q^2+1)$ \\
\hline%
$\begin{array}{c}\\[-2mm]\begin{pmatrix}
	\tau^b&&1& \\ &\tau^{bq}&&1 \\ &&\tau^b& \\ &&&\tau^{bq}
\end{pmatrix}\\ \\[-2mm]\end{array}$
& $\begin{array}{c} q-1\mid b, \\[1.5mm] b\not\equiv 0,\frac{q^2-1}{2},\pm 2(q-1) \pmod{q^2-1}\end{array}$ & $q(q-1)(q+1)^2(q^2+1)$ \\
\hline%
$\begin{array}{c}\\[-2mm]\begin{pmatrix}
	\tau^{2(q-1)}&&& \\ &\tau^{-2(q-1)}&& \\ &&\tau^{b}& \\ &&&\tau^{bq}
\end{pmatrix}\\ \\[-2mm]\end{array}$
& $\begin{array}{c}
		q-1\mid b, \\[1.5mm]
		b\not\equiv 0,\frac{q^2-1}{2},\pm 2(q-1) \pmod{q^2-1}
	\end{array}$ & $q^2(q+1)(q^2+1)$ \\
\hline%
$\begin{array}{c}\\[-2mm]\begin{pmatrix}
	\tau^b&&& \\ &\tau^{bq}&& \\ &&\tau^{b'}& \\ &&&\tau^{b'q}
\end{pmatrix}\\ \\[-2mm]\end{array}$
& $\begin{array}{c}
		q-1\mid b,b',\;\text{and}\;b\not\equiv b',b'q\pmod{q^2-1}, \\[1.5mm]
		b,b'\not\equiv 0,\frac{q^2-1}{2},\pm 2(q-1) \pmod{q^2-1}
	\end{array}$ & $q^2(q+1)^2(q^2+1)$ \\
\hline%
$\begin{array}{c}\\[-2mm]\begin{pmatrix}
	\tau^b&&& \\ &\tau^{bq}&& \\ &&\tau^{-b}& \\ &&&\tau^{-bq}
\end{pmatrix}\\ \\[-2mm]\end{array}$
& $q\pm 1\nmid b$ & $q^2(q-1)(q+1)(q^2+1)$ \\
\hline%
$\begin{array}{c}\\[-2mm]\begin{pmatrix}
	\omega^d&&& \\ &\omega^{dq}&& \\ &&\omega^{dq^2}& \\ &&&\omega^{dq^3}
\end{pmatrix}\\ \\[-2mm]\end{array}$
& $\begin{array}{c} q^2-1\mid d, \\[1.5mm] d\not\equiv 0,\frac{q^4-1}{2}\pmod{q^4-1}\end{array}$ & $q^2(q-1)(q+1)(q^2+1)$ \\
\hline
\hline
\multicolumn{3}{l}{Here, $\rho$, $\tau$ and $\omega$ are primitive elements of $\mathbb{F}_q$, $\mathbb{F}_{q^2}$ and $\mathbb{F}_{q^4}$, respectively.}
\end{longtable}
\end{center}
The number of the conjugacy classes that intersect $\Omega$ is $q^2+q-1=\rank(G/K)-1$.
If it was equal to $\rank(G/K)$, then Condition ($\bigstar$) would be satisfied, but this is not the case for this particular $\alpha$.
Compared with the decomposition of the permutation character given in Theorem 1, only the conjugacy class corresponding to the Jordan canonical form $\left(\begin{smallmatrix} \tau^{2(q-1)} && 1 & \\ & \tau^{-2(q-1)} && 1 \\ && \tau^{2(q-1)} & \\ &&& \tau^{-2(q-1)} \end{smallmatrix}\right)$ is missing in the above table.

\section{Concluding remarks}
We believe that Condition ($\bigstar$) is an interesting condition on Gelfand pairs, and we believe that we should try to classify Gelfand pairs which satisfy Condition ($\bigstar$).
There are some reasons we believe that Condition ($\bigstar$) has something to do with the facts that sometimes the character tables of some big association schemes are controlled by the character tables of smaller association schemes, as it was discussed in \cite{Bannai}.
We speculate that Condition ($\bigstar$) may be relevant to these phenomena, but we will not go further about this in this preliminary report.